\documentclass[10pt]{amsart}
\usepackage{amsmath,amssymb}
\usepackage{bbm}
\usepackage{graphicx,tikz,mathtools}
\usepackage[left=2.65cm,right=2.65cm,top=2.7cm,bottom=2.7cm]{geometry}
\usepackage[pdftex]{hyperref}
\usepackage{cite}
\usepackage{mathrsfs} 
\usepackage{caption}
\hypersetup{
	colorlinks=true,
	linkcolor=blue, 
	citecolor=blue,
	filecolor=blue,
	urlcolor=blue,
}

\makeatletter
\@namedef{subjclassname@2020}{\textup{2020} Mathematics Subject Classification}
\makeatother

\newcommand{\bs}[1]{\boldsymbol{#1}}

\newtheorem{theorem}{Theorem}[section]
\newtheorem{proposition}[theorem]{Proposition}

\newtheorem{lemma}[theorem]{Lemma}
\newtheorem*{conjecture*}{Conjecture}

\newcommand{\triple}[1]{{\left\vert\kern-0.25ex\left\vert\kern-0.25ex\left\vert #1
        \right\vert\kern-0.25ex\right\vert\kern-0.25ex\right\vert}}

\theoremstyle{definition}

\newtheorem{remark}[theorem]{Remark}

\newcommand{\al}{\alpha}
\newcommand{\be}{\beta}
\newcommand{\de}{\delta}
\newcommand{\ep}{\varepsilon}

\newcommand{\ga}{\gamma}

\newcommand{\te}{\theta}

\newcommand{\Om}{\Omega}

\def\NN{\mathbb{N}}
\def\RR{\mathbb{R}}

\def\ZZ{\mathbb{Z}}

\def\bL{\bs{\mathcal{L}}}
\def\bE{\bs{\mathcal{E}}}
\def\bN{\bs{\mathcal{N}}}
\def\cc{\mathbf{c}}

\newcommand{\cE}{{\mathcal E}}

\newcommand{\cI}{{\mathcal I}}
\newcommand{\cK}{{\mathcal K}}
\newcommand{\cJ}{{\mathcal J}}
\newcommand{\cL}{{\mathcal L}}

\newcommand{\cN}{{\mathcal N}}

\newcommand{\pd}{\partial}
\newcommand\minus\backslash

\newcommand\lan\langle
\newcommand\ran\rangle

\newcommand{\Span}{\operatorname{span}}

\newcommand{\e}{{\mathrm e}}

\renewcommand\leq\leqslant
\renewcommand\geq\geqslant

\newcommand\loc{_{\mathrm{loc}}}

\numberwithin{equation}{section}

\begin{document}
 
\title[Blowing-up solutions to competitive critical systems in dimension 3]{Blowing-up solutions to \\ competitive critical systems in dimension 3}

\author[A. J. Fern\'andez]{Antonio J.\ Fern\'andez}
\address{ \vspace{-0.4cm}
\newline 
\textbf{{\small Antonio J. Fern\'andez}} 
\vspace{0.15cm}
\newline \indent Departamento de Matem\'aticas, Universidad Aut\'onoma de Madrid, 28049 Madrid, Spain}
 \email{antonioj.fernandez@uam.es}

\author[M. Medina]{Maria Medina}
\address{ \vspace{-0.4cm}
\newline 
\textbf{{\small Maria Medina}} 
\vspace{0.15cm}
\newline \indent  
Departamento de Matem\'aticas, Universidad Aut\'onoma de Madrid, 28049 Madrid, Spain}
 \email{maria.medina@uam.es}
 
\author[A. Pistoia]{Angela Pistoia}
\address{ \vspace{-0.4cm}
\newline 
\textbf{{\small Angela Pistoia}} 
\vspace{0.15cm}
\newline \indent  
Dipartimento di Scienze di Base e Applicate per l'Ingegneria, Sapienza Universit\`a di Roma, 00161 Roma, Italy}
\email{angela.pistoia@uniroma1.it}



%
%
\begin{abstract}  
We study the critical system of $m\geq 2$ equations
\begin{equation*} 
-\Delta u_i =  u_i^5+\sum_{j = 1,\,j\neq i}^m \beta_{ij}  u_i^2  u_j^3\,, \quad u_i \gneqq 0 \quad \textup{in } \RR^3\,, \quad  i \in \{1, \ldots, m\}\,,
\end{equation*}
where $\beta_{\kappa\ell} =\alpha\in\RR$ if $\kappa\neq\ell$, and $\beta_{\ell m}=\beta_{m \kappa} =\beta<0$, for $ \kappa, \ell \in \{1,\ldots, m-1\}$. We construct solutions to this system in the case where $\beta\to-\infty$ by means of a Ljapunov-Schmidt reduction argument. This allows us to identify the explicit form of the solution at main order:  $u_1$ will look like a perturbation of the standard radial positive solution to the Yamabe equation, while $u_2$ will blow-up at the $k$ vertices of a regular planar polygon. The solutions to the other equations will replicate the blowing-up structure under an appropriate rotation that ensures $u_i\neq u_j$ for $i\neq j$. The result provides the first almost-explicit example of non-synchronized solutions to competitive critical systems in dimension 3. \vspace{-0.5cm}
\end{abstract}

\maketitle
 
\section{Introduction} \label{S.Introduction}
Let us consider the \textit{critical} system
\begin{equation} \label{E.systemGeneralIntro}
-\Delta u_i =  \sum_{j = 1}^m \beta_{ij}  u_i^2  u_j^3\,, \quad u_i \gneqq 0 \quad \textup{in } \RR^3\,, \quad  i \in \{1, \ldots, m\}\,,
\end{equation}
where $m \geq 2$ is an integer and $\be_{ij} \in \RR$ for all $i,j \in \{1, \ldots,m\}$. This kind of systems naturally arise when looking for standing waves in critical systems of Schr\"odinger equations, see
 e.g. \cite{AA1999,T1998}. From the physical point of view, $\be_{ii}$ discriminates between focusing and defocusing behavior of a single component $u_i$, and $\be_{ij}$, $i \neq j$, describes the interspecies forces between particles of different components $u_i$ and $u_j$. If $\be_{ij} < 0$ the particles are in competition and repel each other, and if $\be_{ij} > 0$ they cooperate and attract each other. 

In this paper we are concerned with the focusing case, namely $\be_{ii} > 0$. 
In some cases, it is straightforward to find solutions to this system. Indeed, let us introduce the function
\begin{equation} \label{E.bubbleIntro}
U(x) := \frac{3^\frac14}{(1+|x|^2)^\frac12}\,,
\end{equation}
and recall that all the \textit{positive} solutions to the\textit{ Yamabe} equation
\begin{equation} \label{E.YamabeIntro}
-\Delta u = u^5 \quad \textup{in } \RR^3\,,
\end{equation}
are given by 
\begin{equation} \label{bub}
U_{\de,\xi}(x) := \frac{1}{\sqrt{\de}}\, U \Big( \frac{x-\xi}{\de} \Big) = \frac{3^\frac14 \sqrt{\de}}{(\de^2 + |x-\xi|^2)^\frac12}\,, \quad \textup{for } \de > 0\,,\ \xi \in \RR^3\,.
\end{equation}
Then, setting $u_i := s_i\, U_{\de,\xi}$ for some $\de > 0$ and $\xi \in \RR^3$, system \eqref{E.systemGeneralIntro} reduces to the algebraic system
$$
s_i = \sum_{j=1}^m \be_{ij} s_i^2 s_j^3\,, \quad s_i > 0\,, \quad  i \in \{1, \ldots,m\}\,.
$$
Whenever one is able to solve this algebraic system, one can then obtain solutions to \eqref{E.systemGeneralIntro}. We refer to this kind of solutions as \textit{synchronized} solutions. We are interested in the existence (or not) of finite energy \textit{non--synchronized} solutions to \eqref{E.systemGeneralIntro}. 

Let us turn first to the case where $m = 2$. Assuming $\be_{11} = \be_{22} = 1$ and $\be_{12} = \be_{21} = \be \in \RR$, system \eqref{E.systemGeneralIntro} reduces to
\begin{equation} \label{2m}
\left\{
\begin{aligned}
-\Delta u  & =  u^5 + \beta u^2 v^3 \quad &&\textup{in } \RR^3\,, \\
-\Delta v & =   v^5 + \beta u^3 v^2 \quad &&\textup{in } \RR^3\,, 
\end{aligned}
\right. \qquad u, v \gneqq 0 \quad \textup{in } \RR^3\,.
\end{equation}
In the attractive regime, i.e. when $\beta>0$, Guo and Liu   proved in \cite[Theorem 3.2]{guoliu2008} that any solution to \eqref{2m} is a synchronized solution. In the repulsive regime, i.e. when $\beta<0$, the situation is however different and non-synchronized solutions do exist. In this case, there is a strong connection between solutions to \eqref{2m} and sign-changing solutions to \eqref{E.YamabeIntro}. 
This was first pointed out in a series of papers by Terracini and her collaborators \cite{CTV2002,NTTV2010,STTZ2016,TT2012,TTVW2011}.   In particular, in  \cite[Theorem 1.5]{STTZ2016}, the authors proved that, if $(u_\beta,v_\beta)_\be$ is a family of solutions to \eqref{2m} uniformly bounded in $L^\infty(\mathbb R^3)$, then a segregation phenomenon happens. More precisely, for all $\al \in (0,1)$, it follows that, up to a subsequence,
$$
u_\beta\to w_+\quad  \hbox{and}\quad v_\beta\to w_-\quad \hbox{in } H^1_{loc}(\mathbb R^3)\cap C^{0,\alpha}_{loc}(\mathbb R^3)\,, \quad \textup{as } \be \to - \infty\,.
$$
Here, $w$ is a  sign-changing solution to \eqref{E.YamabeIntro} and, for all $s \in \RR$, $s_+ := \max\{s,0\}$ and $s_- := \max\{-s,0\}$. Having in mind this connection, it is natural to wonder if, whenever there exists a sign-changing solution $w$ to \eqref{E.YamabeIntro}, it is possible to find a solution to \eqref{2m} whose components resemble the positive and negative part of $w$, as $\beta\to-\infty.$
The first result in this direction was obtained by Clapp and Pistoia \cite{CP2018} using variational methods. More precisely, the authors adapted an argument by Ding \cite{D1986}, where he established the existence of infinitely many sign-changing solutions to \eqref{E.YamabeIntro}, and proved the existence of solutions to \eqref{2m} exhibiting a segregation phenomenon as $\be \to - \infty$. Let us point out that, aside from Ding's sign-changing solutions to \eqref{E.YamabeIntro}, there are many others built using perturbative methods \cite{medinamusso2021, delpimupapi2011}. At main order, the sign-changing solutions to \eqref{E.YamabeIntro} constructed by del Pino et al. \cite{delpimupapi2011} are given by
$$
w_k = U - \sum_{j=1}^k U_{\de_k,\, \xi_{j,k}}\,,
$$
where $k$ is a sufficiently large integer, the points $\xi_{j,k}$ are the vertices of a regular polygon placed on a great circle on the $(x_1,x_2)$-plane, and $\de_k \to 0$ as $k \to \infty$. It seems then natural to attempt a similar construction in order to build solutions to \eqref{2m}. More precisely, one may try to construct solutions to \eqref{2m} whose components resemble respectively $U$ and $\sum_{j} U_{\de_k,\,\xi_{j,k}}$. A first attempt in this direction is due to Guo et al. \cite{GuoLiWei2014}. Unfortunately, the proof contains a gap, and their construction does not seem to be feasible. In \cite[Proposition 2.3]{GuoLiWei2014}, the authors claim to be able to invert the corresponding linearized operator (the analogue to Proposition \ref{P.linear} in this work) as a consequence of the symmetries of the construction, without the need of projecting onto any space. However, this is not the case. The chain of identities in \cite[Equation (2.25)]{GuoLiWei2014} does not hold, and they cannot conclude the proof of \cite[Proposition 2.3]{GuoLiWei2014}. 

Concerning the case where $m \geq 3$, there are very few results in the literature. The first in this direction was obtained by Clapp et al. \cite{CSS2021}. There, the authors were able to generalize the results obtained by Clapp and Pistoia \cite{CP2018} to systems with $m \geq 2$ components. Let us also mention the very recent result by Clapp et al. \cite{CPT2021} for systems of Yamabe type equations on closed Riemannian manifolds.  In both papers \cite{CSS2021, CPT2021} the solutions do exhibit a segregation phenomenon as the competition parameters go to $- \infty$. 

In this paper we construct a new type of solutions to \eqref{E.systemGeneralIntro} which are non-synchronized and  whose components do not exhibit a segregation phenomenon in a large repulsive regime. Indeed, let us denote by
\begin{equation}\label{E.rotmatrix}
 {\mathscr R}_\theta:=\left(\begin{matrix}\cos\theta&-\sin\theta&0\\
\sin\theta&\cos\theta&0\\
0&0&1 \\
 \end{matrix}\right),\quad \theta\in \RR\,,
\end{equation}
the rotation matrices on the $(x_1,x_2)$-plane.  Also note that, as usual, the space  $\dot{H}^1 := \dot{H}^1(\RR^3)$ denotes the completion of $C_c^{\infty}(\RR^3)$ with respect to the norm  
$$
\|\phi\| := \left(\int_{\RR^3} |\nabla\phi|^2 dx\right)^{\frac12}\,, \quad \textup{for all } \phi \in \dot{H}^1\,.
$$
We can prove the following:

\begin{theorem}\label{T.mainResultIntro} 
Assume that
\begin{equation}\label{E.beta} \beta_{ii}= \be_{mm} = 1,\quad
\beta_{ij} =\alpha\quad \hbox{and}\quad
\beta_{\ell m}=\beta_{m \kappa} =\beta\quad  \hbox{for }\ i,j,\kappa, \ell \in \{1,\ldots, m-1\}\,, \ i \neq j\,,
\end{equation}
for certain $\alpha,\beta\in \mathbb R$ with $\beta<0$. Then, for any fixed integer $k \geq 2$, there exists $\beta_\star < -\sqrt{2}$ such that, for each $\beta \in (-\infty,\beta_\star]$,  \eqref{E.systemGeneralIntro} has a finite energy solution $(u_{1,\beta},u_{2,\beta},\ldots,u_{m,\beta})$ of the form
$$
u_{1,\beta} = U + \phi_\beta \,, \quad u_{2,\beta} = \sum_{j=1}^k \frac{1}{\tilde{\de}_\be^\frac12}\, U \Big( \frac{\,\cdot\,-\xi_{\be,j}}{\tilde{\de}_\beta} \Big) + \psi_\beta\,,\quad u_{i,\beta}=u_{2,\beta}\Big({\mathscr R}_{\frac{(i-2)}{m-1}\frac{2\pi}{k}} \cdot \Big) \quad \textup{for } i \in \{3,\ldots,m\}\,.
$$
Here, $U$ is defined in \eqref{E.bubbleIntro}, ${\mathscr R}_{\frac{(i-2)}{m-1}\frac{2\pi}{k}}$ is given by \eqref{E.rotmatrix}, $\phi_\beta, \psi_\beta \in \dot{H}^1$ and $(\xi_{\beta,j})_j \subset \RR^3$ satisfy
$$
\|\phi_\beta\|_{\dot{H}^1} + \|\psi_\beta\|_{\dot{H}^1} \to 0 \quad \textup{ and } \quad  \xi_{\beta,j} \to  \Big( \e^{\frac{2\pi(j-1)}{k}i}, 0 \Big)\,,   \quad \textup{as } \beta \to - \infty\,,
$$
and $\tilde{\de}_\beta \in (0,1/\e^2)$ satisfies 
$$
\tilde{\de}_\be^{\frac12} |\log \tilde{\de}_\be| \sim -\frac{1}{\beta}\,.
$$
\end{theorem}

It is worth pointing out that the solutions constructed in Theorem \ref{T.mainResultIntro} are not uniformly bounded as $\be \to -\infty$. Indeed, the $i-$th components for any $i \in \{2,\dots,m\}$ blow-up at the $k$ vertices of a regular planar polygon. Even more, these solutions do not exhibit a segregation phenomenon as $\be \to -\infty$. In fact, the limiting profile of the first component is $U$ (see \eqref{bub}), which is positive everywhere, and the limiting profile of all the other components is a sum of Dirac delta masses at the $k$ vertices of a regular planar polygon. This shows that the segregation result \cite[Theorem 1.5]{STTZ2016} is somewhat sharp. 

Let us also emphasize the fact that the presence of the parameter $\beta$ is what allows us to keep $k$ fixed. As it will be clear, the strategy used in the manuscript relies on a reduction procedure, where an equlibrium must be found. This is done by adjusting a free parameter, which happens to be $\beta$. This is in big contrast with the case of the single equation, treated in \cite{delpimupapi2011, delpimupapi2013}. There, the absence of a natural parameter in the problem forces the authors to {\it artificially} introduce one, which is the number of peaks $k$.  It is worth noting  that the use of the number of peaks as a parameter, first introduced in \cite{weiyan2010}, has in fact been a winning choice in several problems. 


\subsection{Strategy of the proof} Theorem \ref{T.mainResultIntro} relies on a classical Ljapunov-Schmidt reduction argument. For simplicity, we first discuss the proof in the case where $m = 2$ and then explain how to reduce the general case $m \geq 3$ to a different \textit{nonlocal} system of two components, which can be treated in a similar fashion.

We start of with an approximated solution of the form
\begin{equation}\label{E.V}
(U,V) = \bigg( U,\ \sum_{j=1}^k U_{\tilde{\de}_\be, \xi_{\be,j}} \bigg)\,,
\end{equation}
where $k \geq 2$ is a fixed integer and the coupling parameter $\be \to - \infty$ is a free term in the reduction process. The gist of the proof is to show that, for some $\tilde{\de}_\be$, there exists a true solution to \eqref{2m} of the form
$$
(u,v) = (U + \phi_\be,V+\psi_\be)
$$
with $\phi_\be, \psi_\be \to 0$ in $\dot{H}^1$, as $\be \to - \infty$. We will address the problem by linearizing around $(U,V)$ and solving the corresponding system for $(\phi_\be, \psi_\be)$. It turns out that this is possible if the concentration parameter $\tilde{\de}_\be$ is suitably chosen. 

The choice of $\tilde{\de}_\be$ becomes apparent in the study of the reduced energy \eqref{E.Fbeta}. At main order, one needs to balance the interaction among $U$ and $V$, which is of order
$
-\be \tilde{\de}_\be^{\frac32} |\log \tilde{\de}_\be|\,,
$
and the interaction among the different bubbles in $V$, which is of order
$-\tilde{\de}_\be\,.$ 
This balance provides the rate
$$
\tilde{\de}_\be^{\frac12} |\log \tilde{\de}_\be| \sim -\frac{1}{\beta}\,,
$$
which is rather slow, but sufficient for our purposes. This is something special of the three-dimensional case we are dealing with. 

It is folklore that the existence of blowing-up solutions in critical problems is strongly affected by the dimension. For instance, given a smooth bounded domain $\Om \subset \RR^n$, $n \geq 4$, and a parameter $\ep > 0$, the celebrated Brezis-Nirenberg problem \cite{BN}
\begin{equation}\label{bn}
-\Delta u=|u|^{4\over n-2}u+\varepsilon u\quad \hbox{in } \Omega\,,\quad u=0\quad \hbox{on } \partial\Omega\,,
\end{equation}
possesses positive and sign-changing solutions which blow-up as $\varepsilon\to0^+$ (see e.g. \cite{Rey,MP2002,MiP2004}).
On the other hand, if $n=3$ and $\Om$ is star-shaped, the problem does not have any positive solutions when $\varepsilon>0$ is small enough \cite{BN}. Moreover, the existence (or not) of sign-changing solutions to \eqref{bn} with $\ep > 0$ small is a challenging open problem\footnote{\cite{WeiBN} was finished after the completion of this manuscript. In \cite{WeiBN} the authors address precisely the Brezis-Nirenberg problem in dimension $n = 3$.} in this case.  

To elucidate more the role of the dimension in our construction, let us compare our result with the one by Chen et al. \cite{ChMePi2023}. To that end, let us consider the critical system in general dimension, which reads as:
\begin{equation} \label{4D}
-\Delta u_i =  \sum_{j = 1}^m \beta_{ij}  u_i^{\frac{2}{n-2}}  u_j^{\frac{n}{n-2}}\,, \quad u_i \gneqq 0 \quad \textup{in } \RR^n\,, \quad  i \in \{1, \ldots, m\}.
\end{equation}
First, let us stress that, when all the coupling parameters are negative, solutions exhibiting segregation were found in \cite{CSS2021,CP2018} using variational methods. On the other hand, for $n = 4$, Druet and Hebey  proved in \cite[Proposition 3.1]{DH2009} that any non-trivial solution to \eqref{4D} is a synchronized solution if all the $\beta_{ij}$ are positive and equal. Here, we want to address the attention to the construction in \cite{ChMePi2023}, that considers again the case $n=4$. There, the authors built solutions to \eqref{4D} in a small competitive regime using a Ljapunov-Schmidt reduction argument. More precisely, assuming \eqref{E.beta}, the authors built solutions to \eqref{4D} for $\be < 0$ with $|\beta|$ sufficiently small. In the case where $m=2$, the first component of their solutions looks like $U$ and the second one like $V$ (see \eqref{E.V}), so our ansatz is very similar to theirs. However, in our case $\be \to - \infty$, while in theirs $\be \to 0^{-}$. If one tries to mimic our computations in the four-dimensional case when $\be \to - \infty$, they would obtain a rate $|\log \tilde{\de}_\be| \sim -\frac{1}{\be}$, which is too slow for the reduction method to succeed. This striking difference is a purely dimensional phenomenon, which highlights one more time the role of the dimension in the existence of blowing-up solutions in critical problems. Actually, the same is expected to hold in dimensions $n\geq 5$. Formal computations show that in high dimensions the reduction process also forces $\beta \to 0^-$. However, this case is tricky because of the nature of the interaction term. The first power becomes sublinear, and hence no longer contractive. This prevents us to apply fixed point-type arguments in a direct way as in the cases $n=3,4$. The problem is being treated in an ongoing work, where we expect to overcome this technical difficulty. 

We now explain how to reduce the case $m \geq 3$ to a \textit{nonlocal} system with only two components. Let $ {\mathscr R}_\theta$ with $\theta\in\RR$ be given in \eqref{E.rotmatrix} and $r,k\in\NN$.  Denote $q=m-1$. We set
\begin{equation}\label{E.rotationIntro}
{\mathscr R}_{r,k}:={\mathscr R}_{\frac{(r-1)}{q}\frac{2\pi}{k}}\,,
\end{equation}
and consider the \textit{nonlocal} system 
\begin{equation}\label{2222Intro}
\left\{
\begin{aligned}
&-\Delta u=u^5+\beta u^2 v^3 + \be u^2  \sum\limits_{r=2}^qv_r^3\quad && \hbox{in}\ \mathbb R^3\,,\\
&-\Delta v=v^5+\beta u^3 v^2+\alpha v^2 \sum\limits_{r=2}^qv_r^3\quad && \hbox{in}\ \mathbb R^3\,,
\end{aligned}
\right. \qquad \qquad u, v \gneqq 0 \quad \textup{in } \RR^3\,,
\end{equation}
where 
\begin{align*} 
v_r(x):=v({\mathscr R}_{r,k} x) \quad \hbox{for } r \in \{1,\dots,q\}\,.
\end{align*}
Then, using \eqref{E.beta}, one can check that, if $(u,v)$ is a solution to \eqref{2222Intro} and satisfies 
\begin{equation*} 
u(x)=u({\mathscr R}_{r,k}x)\quad \hbox{for all } r \in \{2,\dots,q\}\,, \quad  \hbox{and}\quad  v(x)=v(\mathscr{R}_{\frac{2\pi}{k}}x)\,,
\end{equation*}
the vector $(u_1,\dots, u_{q+1})$, with 
\begin{align*}
u_i(x)=v({\mathscr R}_{i,k} x)\quad \hbox{for all}\ i\in \{1,\dots,q\}\,, \quad \hbox{and}\quad u_{q+1}(x)=u(x)\,, 
\end{align*}
is a solution to \eqref{E.systemGeneralIntro}. We refer to Sect. \ref{S.generalcase} for more details. Let us stress that the assumption \eqref{E.beta} is crucial in the reduction. If we remove it, the reduction cannot be done. The key fact here is that the nonlocal terms appearing in \eqref{2222Intro} are of lower order with respect to the ones already present in \eqref{2m}. Using the symmetries above, we can then deal with \eqref{2222Intro} in a similar fashion as we did with \eqref{E.systemGeneralIntro}. In particular, this allows the parameter $\al \in \RR$ to be freely chosen. 
 
This reduction is inspired by the one in \cite[Sect. 3]{ChMePi2023}. However, the symmetries we are considering are different from the ones there. In particular, we do not need to impose that $k$ is even as they did. Actually, we strongly believe that our approach can lead to the construction of two new different families of solutions in the four-dimensional case when $\be \to 0^-$. Moreover, this would allow to treat the case where $k$ is odd and would substantially simplify the proofs in \cite{ChMePi2023}.

\subsection{Organization of the paper} In Section \ref{S.m=2} we consider the case of two equations; i.e., we prove Theorem \ref{T.mainResultIntro} when $m=2$, following a Ljapunov-Schmidt reduction structure. In Section \ref{S.generalcase} we prove the theorem in the general case $m\geq 3$. The first part of the section is devoted to reduce the system of $m$ equations to a new one of only two (see \eqref{2222Intro}), and the second part to the solvability of this new system.  

\section{The case $m = 2$} \label{S.m=2}
In this section we will prove Theorem \ref{T.mainResultIntro} in the case of two equations. Taking into account the assumptions there, \eqref{E.systemGeneralIntro} with $m = 2$ reduces to
\begin{equation} \label{E.system2eq}
\left\{
\begin{aligned}
-\Delta u & = u^5 + \beta u^2 v^3 \quad &&\textup{in } \RR^3\,, \\
-\Delta v & = v^5 + \beta u^3 v^2 \quad &&\textup{in } \RR^3\,, 
\end{aligned}
\right. \qquad u, v \gneqq 0 \quad \textup{in } \RR^3\,.
\end{equation}
Hence we can reformulate Theorem \ref{T.mainResultIntro} as follows:

\begin{theorem}[The case $m = 2$]  \label{T.mainResult2eq}
For any fixed integer $k \geq 2$, there exists $\beta_\star < -\sqrt{2}$ such that, for each $\beta \in (-\infty,\beta_\star]$,  \eqref{E.system2eq} has a finite energy solution $(u_\beta,v_\beta)$ of the form
$$
u_\beta = U + \phi_\beta \,, \quad v_\beta = \sum_{j=1}^k \frac{1}{\tilde{\de}_\be^\frac12}\, U \Big( \frac{\,\cdot\,-\xi_{\be,j}}{\tilde{\de}_\beta} \Big) + \psi_\beta\,.
$$
Here, $\phi_\beta, \psi_\beta \in \dot{H}^1$ and $(\xi_{\beta,j})_j \subset \RR^3$ satisfy
$$
\|\phi_\beta\|_{\dot{H}^1} + \|\psi_\beta\|_{\dot{H}^1} \to 0 \quad \textup{ and } \quad  \xi_{\beta,j} \to  \Big( \e^{\frac{2\pi(j-1)}{k}i}, 0 \Big)\,,   \quad \textup{as } \beta \to - \infty\,,
$$
and $\tilde{\de}_\beta \in (0,1/\e^2)$ satisfies 
$$
\tilde{\de}_\be^{\frac12} |\log \tilde{\de}_\be| \sim -\frac{1}{\beta}\,.
$$
\end{theorem}

Here, we recall that the space  $\dot{H}^1 := \dot{H}^1(\RR^3)$, as usual, denotes the completion of $C_c^{\infty}(\RR^3)$ with respect to the norm  
$$
\|\phi\| := \left(\int_{\RR^3} |\nabla\phi|^2 dx\right)^{\frac12}\,, \quad \textup{for all } \phi \in \dot{H}^1\,.
$$
Note that $\dot{H}^1$ is a Hilbert space with the scalar product
$$
\langle \phi, \psi \rangle := \int_{\RR^3} \nabla \phi \nabla \psi\, dx\,, \quad \textup{for all } \phi, \psi \in \dot{H}^1\,.
$$

Since we are looking for nonnegative solutions, instead of dealing directly with \eqref{E.system2eq}, we construct a solution to 
\begin{equation} \label{E.system2eq+}
\left\{
\begin{aligned}
-\Delta u & = u_+^5 + \beta u_+^2 v_+^3 \quad &&\textup{in } \RR^3\,, \\
-\Delta v & = v_+^5 + \beta u_+^3 v_+^2 \quad &&\textup{in } \RR^3\,,
\end{aligned}
\right.
\end{equation}
where, for $s \in \RR$, we denote $s_{+} := \max\{s,0\}$ and $s_{-} := \max\{-s,0\}$. It is easy to check that, if $(u,v) \in \dot{H}^1 \times \dot{H}^1$ is a solution to \eqref{E.system2eq+}, then $(u,v)$ is nonnegative and so it is a solution to \eqref{E.system2eq}. Indeed, one just have to use $u_-$ as test function in the first equation in \eqref{E.system2eq+} and $v_-$ in the second one to get $u_-\equiv v_-\equiv 0$ in $\RR^3.$ Hence, from now on we deal with \eqref{E.system2eq+} instead of \eqref{E.system2eq}.


First of all, we recall that
\begin{equation} \label{E.bubble}
U(x) := \frac{3^\frac14}{(1+|x|^2)^\frac12}\,,
\end{equation}
and, for $\de > 0$ and $\xi \in \RR^3$, we set 
\begin{equation} \label{E.bubbleDeltaXi}
U_{\de,\xi}(x) := \frac{1}{\sqrt{\de}} U \Big( \frac{x-\xi}{\de} \Big) = \frac{3^\frac14 \sqrt{\de}}{(\de^2 + |x-\xi|^2)^\frac12}\,.
\end{equation}
Likewise, for $t >0$, $\be < - \sqrt{2}$,
 $k \in \ZZ$ with $k \geq 2$, and $j \in \{1,\ldots,k\}$, with a small abuse of notation we set
\begin{equation} 
U_{t,j}(x):= U_{t \de_\be, \, \xi_{t,j}}(x) = \frac{3^\frac14 \sqrt{t \de_\be}}{(t^2\de_\be^2 + |x-\xi_{t,j}|^2)^\frac12}\,, \quad \textup{with} \quad \xi_{t,j} := \sqrt{1-t^2 \delta_\be^2}\, \Big( \e^{\frac{2\pi(j-1)}{k}i}, 0 \Big)\,.
\end{equation}
Here, $\de_\be \in (0,1/\e^2)$ is chosen so that
\begin{equation} \label{E.delta_beta}
\de_\be^\frac12 |\log {\de_\be}| = - \frac{1}{\beta }\,.
\end{equation}
Finally, we set
\begin{equation} \label{E.sumofbubbles}
V(x) := \sum_{j=1}^k U_{t,j}(x)\,.
\end{equation}
Let us stress that $V$ depends on $t,$ $k$ and $\beta$ but that, for simplicity, we will omit the explicit  dependence. 

Next, for every $k \in \ZZ$ with $k \geq 2$, we let $X_k$ be the subset of $\dot{H}^1$ such that:
\begin{itemize} 
\item Every $\phi \in X_k$ is even with respect to $x_2$ and $x_3$, i.e. 
\begin{equation} \label{E.invarianceEven}
\phi(x_1,x_2,x_3) = \phi(x_1,-x_2,x_3) = \phi(x_1,x_2,-x_3)\,, \quad \textup{for all }  (x_1,x_2,x_3) \in \RR^3\,.
\end{equation}
\item Every $\phi \in X_k$ is invariant under rotation of $2\pi/k$ in the $x_1,x_2$-variables, i.e.
\begin{equation} \label{E.invarianceRotation}
\phi(r\e^{i\te},x_3) = \phi(r \e^{(\te + \frac{2\pi}{k})i}, x_3)\,, \quad \textup{for all } \te, x_3 \in \RR \textup{ and } r > 0\,.
\end{equation}
\item Every $\phi \in X_k$ is invariant under the action of the Kelvin transform, i.e.
\begin{equation} \label{E.invarianceKelvin}
\phi(x) = \frac{1}{|x|} \phi \Big( \frac{x}{|x|^2} \Big) \,, \quad \textup{for all } x \in \RR^3 \setminus \{0\}\,.
\end{equation}
\end{itemize}
In other words, we define
\begin{equation} \label{E.Xk}
X_k := \big\{ \phi \in \dot{H}^1 : \phi \textup{ satisfies  \eqref{E.invarianceEven}, \eqref{E.invarianceRotation} and \eqref{E.invarianceKelvin}} \big\}\,.
\end{equation}
Note that $X_k$ is a Hilbert space endowed with the same scalar product and the same norm as $\dot{H}^1$.
By Sobolev inequality, we have that the embedding $\dot{H}^1 \hookrightarrow L^6 (\RR^3)$ is continuous. Hence we can define, via the Riesz representation theorem, the continuous operator 
$$
(-\Delta)^{-1} : L^{\frac65}(\RR^3) \to \dot{H}^1\,, \quad f \mapsto u\,,  
$$
where $u \in \dot{H}^1$ is the unique solution to
$$
-\Delta u = f \quad \textup{in } \RR^3\,.
$$
Thus, we can reformulate \eqref{E.system2eq+} as
\begin{equation} \label{E.system2eq+inv}
\left\{
\begin{aligned}
u & = (-\Delta)^{-1}(u_+^5 + \beta u_+^2 v_+^3) \quad &&\textup{in } \RR^3\,, \\
v & = (-\Delta)^{-1}(v_+^5 + \beta u_+^3 v_+^2) \quad &&\textup{in } \RR^3\,,
\end{aligned}
\right.
\end{equation}
and we will look for a solution to \eqref{E.system2eq+inv} of the form
\begin{equation} \label{E.ansatz}
u = U + \phi \,, \quad v = V + \psi\,,
\end{equation}
with $\phi, \, \psi \in X_k$ sufficiently small. It is worth emphasizing that, if $u,v \in X_k$, then also the right hand side in \eqref{E.system2eq+inv} belongs to $X_k$. In particular, let us point out that
$$
\widetilde{\phi}(x) := \frac{1}{|x|} \phi \Big( \frac{x}{|x|^2} \Big) \quad \Longrightarrow  \quad  \Delta \phi (y) = \frac{1}{|y|^5} \, \Delta \widetilde{\phi} \Big( \frac{y}{|y|^2} \Big) \,.
$$

In terms of the functions $\phi$ and $\psi$ we can rewrite \eqref{E.system2eq+} as
\begin{equation} \label{E.systemphipsi1}
\left\{
\begin{aligned}
\mathcal{L}_1\phi & = \cE_1 + \cN_1(\phi,\psi)\,, \\
\mathcal{L}_2\psi & = \cE_2 + \cN_2(\phi,\psi)\,,
\end{aligned}
\right.
\end{equation}
where the linear operators $(\cL_1, \cL_2)$ are given by
\begin{equation}\label{E.linear}
\cL_1\phi:= -\Delta \phi - 5 U^4 \phi\,, \quad \cL_2 \psi := -\Delta \psi - 5 V^4 \psi\,,
\end{equation}
the error terms $(\cE_1, \cE_2)$ are given by
\begin{equation} \label{E.error}
\cE_1 := \beta U^2 V^3\,, \quad \cE_2 := V^5 - \sum_{j=1}^k U_{t,j}^5 + \beta U^3 V^2\,,
\end{equation}
and the coupling terms $(\cN_1,\cN_2)$ are given by
\begin{equation}\label{E.nonlinearCoupling}
\begin{aligned}
& \cN_1(\phi,\psi) := (U+\phi)_+^5-U^5-5U^4\phi+\beta(U+\phi)_+^2(V+\psi)_+^3-\beta U^2V^3\,, \\
& \cN_2(\phi,\psi) := (V+\psi)_+^5-V^5-5V^4\psi+\beta(V+\psi)_+^2(U+\phi)_+^3-\beta V^2U^3\,.
\end{aligned}
\end{equation}
Using the dual formulation,  \eqref{E.systemphipsi1} is equivalent to:
\begin{equation} \label{E.systemphipsiDual}
\bL^*(\phi,\psi) - \bE^* - \bN^*(\phi,\psi) = 0 \,, \quad (\phi,\psi) \in X_k \times X_k\,.
\end{equation} 
Here, the linear operator $\bL^* := (\cL_1^*, \cL_2^*)$ is given by
$$
\cL_1^*\phi := \phi - (-\Delta)^{-1} (5U^4 \phi) \,, \quad \cL_2^* \psi := \psi - (-\Delta)^{-1}(5V^4 \psi)\,,
$$
the error term $\bE^* := (\cE_1^*,\cE_2^*)$ is given by
$$
\cE_1^* := (-\Delta)^{-1} \cE_1  \,, \quad \cE_2^* := (-\Delta)^{-1} \cE_2\,,
$$
and the coupling term $\bN^* := (\cN_1^*, \cN_2^*)$ is given by
$$
\cN_1^*(\phi,\psi) := (-\Delta)^{-1} \cN_1(\phi,\psi)\,, \quad \cN_2^*(\phi,\psi) := (-\Delta)^{-1} \cN_2(\phi,\psi)\,. 
$$

The first difficulty when trying to find a solution to \eqref{E.systemphipsiDual} is that the operator $\bL^*$ is in general not invertible. We will overcome this issue using a Ljapunov-Schmidt reduction argument. 

We consider the linearization of the Yamabe equation around $U$ (given in \eqref{E.bubble}), i.e., 
\begin{equation} \label{E.linearizedYamabe}
-\Delta \phi - 5 U^4 \phi = 0 \quad \textup{in } \RR^3\,.
\end{equation}
It is well known (see e.g. \cite{BiEg1991}) that the set of $\dot{H}^1$-solutions to \eqref{E.linearizedYamabe} is $\Span\{Z^{(\ell)}: \ell \in \{0,\ldots,3\} \}$, where
\begin{equation}\label{E.kernelYamabe}
Z^{(0)}(x) := \frac{3^{\frac14} (|x|^2-1)}{2(1+|x|^2)^\frac32}\,, \qquad Z^{(\ell)}(x) := \frac{3^\frac14\, x_\ell}{(1+|x|^2)^\frac32} \quad \ell \in \{1,2,3\}\,. 
\end{equation}
Similarly to \eqref{E.bubbleDeltaXi} and \eqref{E.sumofbubbles}, for $\de > 0$ and $\xi \in \RR^3$, we set
$$
Z_{\de,\xi}^{(0)}(x) := \frac{1}{\sqrt{\de}}\,  Z^{(0)} \Big( \frac{x-\xi}{\de} \Big) = \frac{3^{\frac14} \sqrt{\de}\,  (|x-\xi|^2-\de^2)}{2  (\de^2+|x-\xi|^2)^\frac32}\,,
$$
and define
\begin{equation}\label{E.Z}
Z_{k,t}(x):= \sum_{j=1}^k Z^{(0)}_{t \de_\be, \xi_{t,j}}(x) =  \frac{3^{\frac14} \sqrt{t \de_\be}}{2} \sum_{j=1}^k \frac{|x-\xi_{t,j}|^2 - t^2 \de_\be^2}{(t^2 \de_\be^2 + |x - \xi_{t,j}|^2)^{\frac32}}\,.
\end{equation}
Having at hand this notation, we set
\begin{equation}\label{E.Kk}
\cK_{k,t} := X_k \cap {\rm span}\{Z_{k,t}\}\,, \qquad \cK_{k,t}^{\perp} := \big\{ \phi \in X_k : \langle \phi, Z_{k,t} \rangle = 0 \big\}\,,  
\end{equation}
and consider the orthogonal projections 
$$
\Pi: X_k \times X_k \to X_k \times \cK_{k,t} \,, \qquad \Pi^{\perp}: X_k \times X_k \to X_k \times \cK_{k,t}^{\perp}\,.
$$
Note that $\cK_{k,t}^{\perp}$ is the space of $X_k$ functions $\dot{H}^1$-orthogonal to $\cK_{k,t}$, and that one has the decomposition
$$
X_k = \cK_{k,t} \oplus \cK_{k,t}^{\perp}\,.
$$
We can then rewrite \eqref{E.systemphipsiDual} as the equivalent system
\begin{equation} \label{E.systemLyapunov-Schmidt}
\left\{
\begin{aligned}
\ & \Pi \big[ \bL^*(\phi,\psi) - \bE^* - \bN^*(\phi,\psi)\big] = 0 \,, \\
& \Pi^{\perp} \big[\bL^*(\phi,\psi) - \bE^* - \bN^*(\phi,\psi)\big] = 0 \,, 
\end{aligned}
\right.
\qquad (\phi,\psi) \in X_k \times X_k\,.
\end{equation}
The first equation is the so-called \textit{auxiliary equation} and, as we will see later on, it is finite dimensional. The rest of the section is devoted to prove the existence of a solution $(\phi,\psi)$ to \eqref{E.systemLyapunov-Schmidt}. 

Through the rest of the paper, we will systematically use \cite[Lemma A.1]{PrVe2019}. For convenience, we state this result here with our notation, and for the particular case of dimension $n = 3$. 

\begin{lemma}{\rm( \hspace{-0.12cm}\cite[Lemma A.1]{PrVe2019})} \label{L.PremoselliVetois}
Let 
\begin{equation} \label{E.Omega1}
\Om_1 := \big\{ x \in \RR^3: |x - \xi_{t,1}| < |x - \xi_{t,j}| \textup{ for all } j \in \{2, \ldots,k\} \big\}\,.
\end{equation}
For every $\nu, \ga \geq 0$ such that $\nu + \ga \leq 6$ and all $0 < a < b$, there exits a constant $C > 0$ such that
$$
\int_{\Om_1} U^{6-\nu-\ga}\, U_{t,1}^{\nu} \bigg( \sum_{j=2}^k U_{t,j}\bigg)^{\ga} dx \leq C \Big( \de_\be^{ \frac{\nu + \ga}{2}} f_1(k,\nu,\ga) + (k \log k)^{\ga} f_2(\de_\be,k,\nu,\ga) \Big)\,,
$$
for all $k \geq 2$ and all $t \in [a,b]$. Here,
$$
f_1(k,\nu,\ga)
:= \left\{
\begin{aligned}
& k^{\ga-1} && \textup{if } \nu < 2\,,\\
& k^{\ga-1} (\log k)^{\ga+1} && \textup{if } \nu = 2\,,\\
& k^{\nu + \ga -3} (\log k)^{\ga} && \textup{if } \nu > 2\,,
\end{aligned}
\right.
\quad \textup{ and } \quad 
f_2(\de_\be,k,\nu,\ga) := \left\{
\begin{aligned}
& \de_\be^{ \frac{\nu + \ga}{2}} k^{\nu-3} && \textup{if } \nu < 3\,,\\
& \de_\be^{ \frac{3+\ga}{2}} |\log \de_\be| && \textup{if } \nu = 3\,,\\
& \de_\be^{3 + \frac{\ga-\nu}{2}} && \textup{if } \nu > 3\,. 
\end{aligned}
\right.
$$
\end{lemma}

\subsection{The error of approximation} As a preliminary step, we quantify how well the ansatz $(U,V)$ ``solves'' \eqref{E.systemphipsiDual}, or equivalently \eqref{E.systemLyapunov-Schmidt}. In order to make the notation lighter, we will write $\|\cdot\|_p = \|\cdot \|_{L^p(\RR^3)}$.

\begin{proposition} \label{P.error}
For all fixed integer $k \geq 2$ and all $0 < a < b$, there exists a constant $C > 0$ such that
$$
C^{-1} \Big( \|\cE_1^*\| + \|\cE_2^*\| \Big) \leq \|\cE_1\|_{\frac65} + \|\cE_2\|_{\frac65} \leq C \de_\be (1+|\beta|)\,,
$$
for all $\beta < -\sqrt{2}$ and $t \in [a,b]$. 
\end{proposition}

\begin{proof}
We first prove that there exists $C > 0$ (independent of $\beta$ and $t$) such that
\begin{equation} \label{E.conclussionE1}
\|\cE_1\|_{\frac65} \leq  C \de_\be |\beta|\,.
\end{equation}

First, taking into account the definition of $\Om_1$ (see \eqref{E.Omega1}), we immediately get that 
\begin{align*}
\|\cE_1\|_{\frac65}^{\frac65} & = |\beta|^{\frac65} \int_{\RR^3} U(x)^{6-\frac{18}{5}} \bigg( \sum_{j=1}^k U_{t,j}(x) \bigg)^{\frac{18}{5}} dx = k |\be|^{\frac65} \int_{\Om_1} U(x)^{6-\frac{18}{5}} \bigg( U_{t,1}(x) + \sum_{j=2}^k U_{t,j}(x) \bigg)^{\frac{18}{5}} dx \\
& \leq 2^{\frac{13}{5}} k |\be|^{\frac65} \int_{\Om_1} U(x)^{6-\frac{18}{5}} U_{t,1}(x)^{\frac{18}{5}} dx +  2^{\frac{13}{5}} k |\be|^{\frac65} \int_{\Om_1} U(x)^{6-\frac{18}{5}} \bigg(  \sum_{j=2}^k U_{t,j}(x) \bigg)^{\frac{18}{5}} dx \,.
\end{align*}
Once we have this decomposition, we estimate each term of the right hand side separately. First, applying Lemma \ref{L.PremoselliVetois} with $\nu = 18/5$ and $\ga = 0$, we get
$$
\int_{\Om_1} U(x)^{6-\frac{18}{5}} U_{t,1}(x)^{\frac{18}{5}} dx \lesssim \de_\beta^{\frac65} ( (\de_\beta k)^{\frac35} + 1) \,.
$$
Likewise, by Lemma \ref{L.PremoselliVetois} with $\nu = 0$ and $\ga = 18/5$, it follows that
$$
\int_{\Om_1} U(x)^{6-\frac{18}{5}} \bigg(  \sum_{j=2}^k U_{t,j}(x) \bigg)^{\frac{18}{5}} dx \lesssim \de_\be^{\frac95} (k^{\frac{13}{5}} + k^{\frac35} \log(k)^{\frac{18}{5}} ) \,.
$$
Thus, we infer that
\begin{equation} \label{E.conclussionE1withk}
\|\cE_1\|_{\frac65} \lesssim \de_\be |\be| \left[ \de_\be^{\frac35} \big(k^{\frac85}(1+\log(k)^{\frac{18}{5}}) + k^{\frac{18}{5}} \big)+1 \right]^{\frac56}.
\end{equation}
Note that here, and during the rest of the proof, the implicit constants may depend on $a$ and $b$, but they are independent of $k$, $\beta$ and $t$. Finally, since $k$ is a fixed number, \eqref{E.conclussionE1} immediately follows from \eqref{E.conclussionE1withk}.

Next, we prove that there exists $C > 0$ (independent of $\be$ and $t$) such that
\begin{equation} \label{E.conclussionE2}
\|\cE_2\|_{\frac65} \leq C \de_\be (1+|\be|)\,.
\end{equation}
To that end, we write
$$
\cE_2 = \cE_{2,1} + \cE_{2,2}\,, \quad \textup{ with } \quad \cE_{2,1} := \beta V^2 U^3 \quad \textup{and} \quad \cE_{2,2} := V^5 - \sum_{j=1}^k U_{t,j}^5\,,
$$
and we estimate each $\cE_{2,i}$ separately. 

First, arguing as in the proof of \eqref{E.conclussionE1withk}, we see that
$$
\|\cE_{2,1}\|_{\frac65}^{\frac65} \leq 2^{\frac75} k |\beta|^{\frac65}  \int_{\Om_1} U(x)^{6-\frac{12}{5}} U_{t,1}(x)^{\frac{12}{5}} dx +  2^{\frac{7}{5}} k |\be|^{\frac65} \int_{\Om_1} U(x)^{6-\frac{12}{5}} \bigg(  \sum_{j=2}^k U_{t,j}(x) \bigg)^{\frac{12}{5}} dx \,.
$$
Then, using twice Lemma \ref{L.PremoselliVetois} (first $\nu = 12/5$ and $\ga = 0$ and then $\nu = 0$ and $\ga = 12/5$), we get
\begin{equation} \label{E.conclussionE21withk}
\|\cE_{2,1}\|_{\frac65} \lesssim \de_\be |\be| \left[ k^{\frac25} (1+\log(k)^{\frac{12}{5}}) + k^{\frac{12}{5}} \right]^{\frac56}\,.
\end{equation}

On the other hand, by \cite[Equation (3.7)]{PrVe2019}, we get that 
$$
\|\cE_{2,2}\|_{\frac65}^{\frac65} \lesssim k \int_{\Om_1} U_{t,1}(x)^{\frac{24}{5}}  \bigg(  \sum_{j=2}^k U_{t,j}(x) \bigg)^{\frac{6}{5}} dx + k \int_{\Om_1} \bigg(  \sum_{j=2}^k U_{t,j}(x) \bigg)^6 dx\,.
$$
Using once again Lemma \ref{L.PremoselliVetois}, we estimate each term of the right hand side and obtain that
\begin{equation} \label{E.conclussionE22withk}
\|\cE_{2,2}\|_{\frac65} \lesssim \de_\be  \left[\de_\be^{\frac95} (k^4 \log(k)^{\frac65} + k^6 + k^4 \log(k)^6 ) + k^{\frac{11}{5}} \log(k)^{\frac65} \right]^{\frac56}\,.
\end{equation}
Since $k$ is a fixed number, \eqref{E.conclussionE2} immediately follows from \eqref{E.conclussionE21withk} and \eqref{E.conclussionE22withk}. 

Taking into account that there exists a constant $C > 0$ such that
\begin{equation} \label{E.continuityDeltaminus1}
\|(-\Delta)^{-1} f\| \leq C \|f\|_{\frac65}\,, \quad \textup{ for all } f \in L^{\frac65}(\RR^3)\,,
\end{equation}
the result follows combining \eqref{E.conclussionE1} and \eqref{E.conclussionE2} with the definition of $\cE_i^*$, $i=1,2$. 
\end{proof}

\subsection{The linear theory} Once we have quantified the error of approximation, we start our analysis of \eqref{E.systemLyapunov-Schmidt} by solving the second equation there. More precisely, we deal with
\begin{equation} \label{E.equationPiPerp}
\Pi^{\perp} \big[\bL^*(\phi,\psi) - \bE^* - \bN^*(\phi,\psi)\big] = 0 \,, \quad (\phi,\psi) \in X_k \times X_k\,.
\end{equation}
To that end, let us introduce the shortened notation
$$
\bL^{\perp}_* := \Pi^{\perp} \bL^*\,, \quad \bE^{\perp}_* := \Pi^{\perp} \bE^*\,, \quad \textup{and} \quad \bN^{\perp}_* := \Pi^{\perp} \bN^*\,.
$$

The first step to solve \eqref{E.equationPiPerp}  is to analyze the invertibility of the linear operator $\bL_*^{\perp}$.

\begin{proposition} \label{P.linear}
For all fixed integer $k \geq 2$ and all $0 < a < b$, there exist constants $\be_\star < -\sqrt{2}$ and $C > 0$ such that, for all $\be \in (-\infty, \be_\star]$ and all $t \in [a,b]$,
\begin{equation}
\|(\phi,\psi)\| \leq C \|\bL^{\perp}_*(\phi,\psi)\| \,, \quad \textup{\textit{for all} } (\phi,\psi) \in X_k \times \cK_{k,t}^{\perp}\,.
\end{equation}
In particular, the inverse operator $(\bL^{\perp}_*)^{-1} : X_k \times \cK_{k,t}^{\perp} \to X_k \times \cK_{k,t}^{\perp}$ exists and is continuous. 
\end{proposition}

\begin{proof}
We argue by contraction and assume that there exist sequences $(\be_n)_n \subset (-\infty,0)$, $(t_n)_n \subset [a,b]$ and $(\phi_n, \psi_n)_n \subset X_k \times \cK_{k,t_n}^{\perp}$ such that $\be_n \to - \infty$ and
\begin{equation} \label{E.contradictionLinearTheory}
\bL_*^{\perp} (\phi_n, \psi_n) = (f_n, g_n) \in X_k \times \cK_{k,t_n}^{\perp}\,, \quad \|\phi_n\| + \|\psi_n\| = 1\,, \quad \textup{and} \quad \|f_n\| + \|g_n\| = o(1)\,. 
\end{equation}
More precisely, we assume that $\be_n \to - \infty$, and that there exists a sequence $(c_n)_n \subset \RR$ such that, for all $n \in \NN$, $(\phi_n,\psi_n)$ solves the system
\begin{equation} \label{E.sequencesLinearized}
\left\{
\begin{aligned}
-\Delta \phi_n - 5 U^4 \phi_n & = -\Delta f_n && \quad \textup{in } \RR^3\,,\\
-\Delta \psi_n - 5 V_n^4 \psi_n & = -\Delta g_n - c_n \Delta Z_n && \quad \textup{in } \RR^3\,.
\end{aligned}
\right.
\end{equation}
Here,
$$
V_n(x) := \sum_{j=1}^k U_{t_n \de_{\be_n},\, \xi_{t_n, j}}(x)\,, \qquad Z_n(x) := \sum_{j=1}^k Z_{t_n \de_{\be_n},\, \xi_{t_n, j}}^{(0)}(x)\,,
$$
and we are assuming that
$$
\|\phi_n\| + \|\psi_n\| = 1\,, \quad \textup{and} \quad \|f_n\| + \|g_n\| = o(1) \quad \textup{as } n \to \infty\,.
$$
For later purposes, we introduce the notation
$$
U_{jn}(x) := U_{t_n \de_{\be_n},\, \xi_{t_n, j}}(x) \quad \textup{and} \quad Z_{jn}(x) :=  Z_{t_n \de_{\be_n},\, \xi_{t_n, j}}^{(0)}(x)\,,
$$
so that
$$
V_n = \sum_{j=1}^k U_{jn} \quad \textup{and} \quad Z_n = \sum_{j=1}^k Z_{jn}\,.
$$
We divide the rest of the proof into four steps for the sake of clarity.

\medbreak
\noindent \textit{Step 1.}\textit{ Up to a subsequence, $\phi_n \rightharpoonup 0$ in $\dot{H}^1$ and $\phi_n \to 0$ in $L^p_{\loc}(\RR^3)$ for all $p \in [2,6)$.}

\smallbreak
Since $(\|\phi_n\|)_n$ is bounded, we have the existence of $\phi \in X_k$ such that, up to a subsequence, 
$$
\phi_n \rightharpoonup \phi \quad \textup{in } \dot{H}^1\,, \quad \textup{and} \quad \phi_n \to \phi \quad \textup{in } L^p_{\loc}(\RR^3) \textup{ for all } p \in [2,6)\,.
$$
Moreover, taking into account \eqref{E.sequencesLinearized}, we infer that
$$
-\Delta \phi - 5 U^4 \phi = 0 \quad \textup{in } \RR^3\,.
$$
Hence, by \cite[Lemma A.1]{BiEg1991}, we get that $\phi \in \Span\{ Z^{(\ell)},\ \ell \in \{0,\ldots,3\}\}$. Taking into account the definition of $X_k$, it is then not difficult to conclude that $\phi \equiv 0$ and so the first step. Indeed, note that
$$
\langle \phi, Z^{(\ell)} \rangle = 5 \int_{\RR^3} U^4 Z^{(\ell)}\phi\, dx \,, \quad \forall\ \ell \in \{0, \ldots,3\}\,.
$$
Then, using \eqref{E.invarianceEven} it is straightforward to see that 
$$
\int_{\RR^3} U^4 Z^{(2)} \phi\, dx = \int_{\RR^3} U^4 Z^{(3)} \phi \, dx = 0\,.
$$
Moreover, combining that $\langle \phi, Z^{(2)} \rangle = 0$ with \eqref{E.invarianceRotation} and the fact that $k \geq 2$, we see that
$$
\int_{\RR^3} U^4 Z^{(1)} \phi\, dx = 0\,.
$$
Finally, using \eqref{E.invarianceKelvin}, we get that
$$
\int_{\RR^3} U^4 Z^{(0)} \phi\, dx = 0\,,
$$
and thus $\langle \phi, Z^{(\ell)} \rangle = 0$ for all $\ell \in \{0, \ldots, 3\}$. This implies that $\phi \equiv 0$, as desired. 

\medbreak

\noindent \textit{Step 2.}  \textit{$\|Z_n\|^2 = \frac{15\sqrt{3}\pi^2}{64} k + o(1)$ and $c_n \to 0$ as $n \to \infty$.}

\smallbreak
First of all, using that $Z^{(0)}$ solves \eqref{E.linearizedYamabe}, we get that
\begin{align*}
& \|Z_n\|^2  = \sum_{j=1}^k \int_{\RR^3} |\nabla Z_{jn}|^2 dx + 2 \sum_{j \neq i} \int_{\RR^3} \nabla Z_{jn} \nabla Z_{in} \, dx = 5 \sum_{j=1}^k \int_{\RR^3} U_{jn}^4 Z_{jn}^2\, dx + 10 \sum_{j\neq i}  \int_{\RR^3}  U_{jn}^4 Z_{jn} Z_{in}\, dx \\
& \quad = 5k \int_{\RR^3} U(y)^4 Z^{(0)}(y)^2 dy +  10  \sum_{j\neq i} \int_{\RR^3}  U_{jn}^4 Z_{jn} Z_{in}\, dx = \frac{15\sqrt{3}\pi^2}{64} k +  10  \sum_{j \neq i} \int_{\RR^3}  U_{jn}^4 Z_{jn} Z_{in}\, dx\,.
\end{align*}
Hence, to conclude  that $\|Z_n\|^2 = \frac{15\sqrt{3}\pi^2}{64} k + o(1)$ as $n \to \infty$, we are just missing to prove that
\begin{equation} \label{E.Zno1}
 10  \sum_{j \neq i} \int_{\RR^3}  U_{jn}^4 Z_{jn} Z_{in}\, dx = o(1) \quad \textup{as } n \to \infty\,.
\end{equation}
We start pointing out that, for all $n \in \NN$ and all $j \in \{1, \ldots,k\}$,
\begin{equation} \label{E.pointwise1}
|Z_{jn}| \leq U_{jn} \quad \textup{in } \RR^3\,.
\end{equation}
Likewise, if we assume that $i \neq j$ and choose $r_0 > 0$ small enough so that $B_{r_0}(\xi_{t_n,j}) \cap B_{r_0}(\xi_{t_n,i}) = \emptyset$, it is straightforward to check that
\begin{equation} \label{E.pointwise2}
U_{jn} \leq \frac{3^\frac14 \sqrt{b}}{r_0} \sqrt{\de_{\be_n}} \quad \textup{in } B_{r_0}(\xi_{t_n,i})\,.
\end{equation}
Using these pointwise estimates, one can check that
$$
\bigg|\int_{B_{r_0}(\xi_{t_n,j})} U_{jn}^4 Z_{jn} Z_{in}\, dx \bigg| \lesssim \de_{\be_n}\,, $$
and that
$$\bigg|\int_{B_{r_0}(\xi_{t_n,i})} U_{jn}^4 Z_{jn} Z_{in}\, dx \bigg| + \bigg| \int_{\RR^3 \setminus (B_{r_0}(\xi_{t_n,i}) \cup B_{r_0}(\xi_{t_n,i}))} U_{jn}^4 Z_{jn} Z_{in}\, dx \bigg| \lesssim \de_{\be_n}^3\,.
$$
Note that here and through the proof, the implicit constants (in $\lesssim\,$) may depend on $r_0$, $a$ and $b$, but they are independent of $n$. Thus, since $\be_n \to - \infty$, by \eqref{E.delta_beta}, we conclude that \eqref{E.Zno1} holds and so that
\begin{equation} \label{E.ZnExpansion}
\|Z_n\|^2 = \frac{15\sqrt{3}\pi^2}{64} k + o(1) \quad \textup{as } n \to \infty\,.
\end{equation}

Next, testing the second equation in \eqref{E.sequencesLinearized} with $Z_n$, and using that $Z^{(0)}$ is a solution to \eqref{E.linearizedYamabe}, that $g_n \in \cK_{k,t_n}^{\perp}$ and \eqref{E.ZnExpansion}, we get that
$$
\int_{\RR^3} \bigg( \sum_{j=1}^k U_{jn}^4 Z_{jn} - V_n^4 Z_n \bigg) \psi_n\, dx = c_n \bigg( \frac{3\sqrt{3}\pi^2}{64} k + o(1) \bigg) \quad \textup{as } n \to \infty\,.
$$
Moreover, using the multinomial theorem, and H\"older and Sobolev inequalities, we infer that
\begin{align*}
& \bigg| \int_{\RR^3} \bigg( \sum_{j=1}^k U_{jn}^4 Z_{jn} - V_n^4 Z_n \bigg) \psi_n\, dx \bigg| \lesssim \|\psi_n\| \bigg[ \sum_{j\neq i} \|U_{jn}^4 Z_{in}\|_{\frac65} + \sum_{j \neq i} \sum_{l = 1}^k \Big( \|U_{jn}^3 U_{in} Z_{ln}\|_{\frac65} + \|U_{jn}^2 U_{in}^2 Z_{ln}\|_{\frac65} \Big)\\
& \quad \sum_{j \neq i \neq r} \sum_{l=1}^k \|U_{jn}^2 U_{in} U_{rn} Z_{ln}\|_{\frac65} + \sum_{j \neq i \neq r \neq p} \sum_{l=1}^k \|U_{jn} U_{in} U_{rn} U_{pn} Z_{ln}\|_{\frac65} \bigg]\,. 
\end{align*}
Since $k \geq 2$ is a fixed integer, using the pointwise estimates \eqref{E.pointwise1}--\eqref{E.pointwise2}, and arguing as we did to prove \eqref{E.Zno1}, we estimate each term on the right hand side and conclude that
$$
 \int_{\RR^3} \bigg( \sum_{j=1}^k U_{jn}^4 Z_{jn} - V_n^4 Z_n \bigg) \psi_n\, dx = O(\de_{\be_n}) \quad \textup{as } n \to \infty\,.
$$
Hence, it follows that
$$
 c_n  \frac{3\sqrt{3}\pi^2}{64} k = o(1) \quad \textup{as } n \to \infty\,,
$$
and so that $c_n \to 0$ as $n \to \infty$. This concludes the proof of the second step. 
\medbreak
\noindent \textit{Step 3.} \textit{Up to a subsequence, $\widetilde{\psi}_n(y):= \sqrt{t_n \de_{\be_n}} \psi_n (t_n \de_{\be_n} y + \xi_{t_n,1})$ satisfies $\widetilde{\psi}_n \rightharpoonup 0$ in $\dot{H}^1$ and $\widetilde{\psi}_n \to 0$ in $L^p_{\loc}(\RR^3)$ for all $p \in [2,6)$.}

\smallbreak
First note that $\|\widetilde{\psi}_n\| = \|\psi_n\|$ for all $n \in \NN$. Hence, since $(\|\psi_n\|)_n$ is bounded, there exists $\psi \in \dot{H}^1$ such that, up to a subsequence,
$$
\widetilde{\psi}_n \rightharpoonup \psi \quad \textup{in } \dot{H}^1\,, \quad \textup{and} \quad \widetilde{\psi}_n \to \psi \quad \textup{in } L^p_{\loc}(\RR^3) \textup{ for all } p \in [2,6)\,.
$$
Moreover, taking into account the second step, we infer that
$$
-\Delta \psi - 5 U^4 \psi = 0 \quad \textup{in } \RR^3\,,
$$
and thus, by \cite[Lemma A.1]{BiEg1991}, we get that $\psi \in \Span\{ Z^{(\ell)},\ \ell \in \{0,\ldots,3\}\}$. Note that, by the definition of $\xi_{t_n,1}$ for all $n \in \NN$, $\widetilde{\psi}_n$ satisfies \eqref{E.invarianceEven} for all $n \in \NN$. Thus, arguing as in the first step, we get that
$$
\langle \psi, Z^{(2)} \rangle = \langle \psi, Z^{(3)} \rangle = 0\,.
$$
On the other hand, since $\psi_n \in \cK_{k,t_n}^{\perp} \subset X_k$,  $Z^{(0)}$ is a solution to \eqref{E.linearizedYamabe} and  $\widetilde{\psi}_n \rightharpoonup \psi$ in $\dot{H}^1$,
\begin{align*}
0 & = \langle \psi_n, Z_n \rangle = 5 \int_{\RR^3} \bigg( \sum_{j=1}^k U_{jn}^4 Z_{jn} \bigg) \psi_n \, dx = 5k \int_{\RR^3} U_{1n}^4 Z_{1n} \psi_n \, dx \\
&  = 5k \int_{\RR^3} U(y)^4 Z^{(0)}(y) \widetilde{\psi}_n(y)\, dy = k \langle  \widetilde{\psi}_n, Z^{(0)}  \rangle = k \langle  \psi, Z^{(0)} \rangle + o(1) \quad \textup{as } n \to \infty\,.
\end{align*}
Taking the limit as $n \to \infty$, we then infer that $\langle \psi, Z^{(0)} \rangle = 0$. Thus, if we prove that $\langle \psi, Z^{(1)} \rangle = 0$,  then necessarily $\psi \equiv 0$ and the third step will follow. 

To that end, using \eqref{E.invarianceKelvin}, we get that
\begin{align*}
& \int_{\RR^3} \frac{y_1}{(1+|y|^2)^{\frac72}}\, \widetilde{\psi}_n(y)\, dy =  (t_n \de_{\be_n})^{\frac72} \int_{\RR^3} \frac{x_1-(1-t_n^2 \de_{\be_n}^2)^\frac12 }{(t_n^2 \de_{\be_n}^2 + |x-\xi_{t_n,1}|^2)^{\frac72}} \, \psi_n(x) \, dx \\ 
& \quad =   (t_n \de_{\be_n})^{\frac72} \int_{\RR^3} \frac{z_1 - (1-t_n^2 \de_{\be_n}^2)^{\frac12}|z|^2 }{(t_n^2 \de_{\be_n}^2 + |z - \xi_{t_n,1}|^2)^{\frac72}} \, \psi_n(z)\, dz \\
& \quad = \int_{\RR^3}  \frac{y_1}{(1+|y|^2)^{\frac72}}\, \widetilde{\psi}_n(y)\, dy + \frac{(1- t_n^2 \de_{\be_n}^2)^{\frac12}}{t_n \de_{\be_n}} \int_{\RR^3} \frac{1-|t_n \de_{\be_n} y + \xi_{t_n,1}|^2}{(1+|y|^2)^{\frac72}}\, \widetilde{\psi}_n (y)\, dy \,.
\end{align*}
Thus, it follows that
\begin{align*}
 0 & = \frac{(1- t_n^2 \de_{\be_n}^2)^{\frac12}}{t_n \de_{\be_n}} \int_{\RR^3} \frac{1-|t_n \de_{\be_n} y + \xi_{t_n,1}|^2}{(1+|y|^2)^{\frac72}}\, \widetilde{\psi}_n (y)\, dy \\
&= (1- t_n^2 \de_{\be_n}^2)^{\frac12} t_n \de_{\be_n} \int_{\RR^3} \frac{1-|y|^2}{(1+|y|^2)^{\frac72}}\, \widetilde{\psi}_n(y)\, dy - 2 (1-t_n^2 \de_{\be_n}^2) \int_{\RR^3} \frac{y_1}{(1+|y|^2)^{\frac72}}\, \widetilde{\psi}_n(y)\, dy \,.
\end{align*}
Sending $n \to \infty$, we conclude that
$$
\int_{\RR^3} \frac{y_1}{(1+|y|^2)^{\frac72}}\, \psi(y)\, dy = 0\,,
$$
and so that $\langle \psi, Z^{(1)} \rangle = 0$, as desired. 

\medbreak
\noindent \textit{Step 4. Conclusion of the proof.}
\smallbreak

We are going to prove that $\| \psi_n\| + \|\phi_n \| \to 0$ as $n \to \infty$, reaching a contradiction with \eqref{E.contradictionLinearTheory}. First of all, we test the first equation in \eqref{E.sequencesLinearized} with $\phi_n$, the second one with $\psi_n$, and we sum them both to get that
\begin{align*}
1 = \|\psi_n\| + \|\phi_n\| = 5 \int_{\RR^3} U^4 \phi_n^2 \, dx + \langle \phi_n, f_n \rangle + 5 \int_{\RR^3} V_n^4 \psi_n^2 \, dx + \langle \psi_n, g_n + c_n Z_n \rangle \,.
\end{align*}
Combining \eqref{E.contradictionLinearTheory} with the second step, we infer that
$$
\langle \phi_n, f_n \rangle \to 0 \quad \textup{and} \quad \langle \psi_n, g_n + c_n Z_n \rangle \to 0 \quad \textup{as } n \to \infty\,.
$$
Likewise, using the first step and the decay of $U$, we infer that
$$
\int_{\RR^3} U^4 \phi_n^2 \, dx \to 0 \quad \textup{as } n \to \infty\,.
$$
Hence, we have that
$$
1 = \|\psi_n\| + \|\phi_n\| = 5 \int_{\RR^3} V_n^4 \psi_n^2 \, dx + o(1) \quad \textup{as } n \to \infty\,,
$$ 
and thus, to conclude the proof, we just have to show that 
$$
\int_{\RR^3} V_n^4 \psi_n^2 \, dx  \to 0 \quad \textup{as } n \to \infty\,.
$$
Recalling the definition of $\Om_1$, see \eqref{E.Omega1}, we get that
\begin{align*}
\int_{\RR^3} V_n^4 \psi_n^2 \, dx = k \int_{\Om_1} \bigg( U_{1n} + \sum_{j=2}^k U_{jn} \bigg)^4 \psi_n^2\, dx \leq 8k \int_{\Om_1} U_{1n}^4 \psi_n^2\, dx + 8k \int_{\Om_1} \bigg( \sum_{j=2}^k U_{jn} \bigg)^4 \psi_n^2 \, dx\,.
\end{align*}
One one hand, by Lemma \ref{L.PremoselliVetois} applied with $\nu = 0$ and $\ga = 6$, we get that
$$
8k \int_{\Om_1} \bigg( \sum_{j=2}^k U_{jn} \bigg)^4 \psi_n^2 \, dx \leq 8k \|\psi_n\|_6^2\, \bigg( \int_{\Om_1} \bigg( \sum_{j=2}^k U_{jn} \bigg)^6 dx \bigg)^{\frac23} \lesssim \de_{\be_n}^2 k^3 \Big( k^2 + \log(k)^6 \Big)^{\frac23}\,.
$$
Since $k \geq 2$ is a fixed integer, we conclude that
$$
8k \int_{\Om_1} \bigg( \sum_{j=2}^k U_{jn} \bigg)^4 \psi_n^2 \, dx = O(\de_{\be_n}^2) \quad \textup{as } n \to \infty\,.
$$

On the other hand, note that
$$
8k \int_{\Om_1} U_{1n}^4 \psi_n^2\, dx \leq 8k \int_{\RR^3} U_{1n}^4 \psi_n^2\, dx = 8k \int_{\RR^3} U(y)^4 \widetilde{\psi}_n(y)^2\, dy\,.
$$
Since $k \geq 2$ is a fixed integer, using the third step and the decay of the function $U$, we get that
$$
8k \int_{\Om_1} U_{1n}^4 \psi_n^2\, dx \to 0 \quad \textup{as } n \to \infty\,,
$$
and so that
$$
\int_{\RR^3} V_n^4 \psi_n^2 \, dx \to 0 \quad \textup{as } n \to \infty\,,
$$
as desired. 
\end{proof}

\subsection{Nonlinear theory} Having at hand Propositions \ref{P.error} and \ref{P.linear} we can now give the desired existence result for \eqref{E.equationPiPerp}. More precisely, we prove the following:

\begin{proposition} \label{P.nonlinear}
For all fixed integer $k \geq 2$ and all $0 < a < b$, there exist constants $\be_{\star} < -\sqrt{2}$ and $C > 0$ such that, for all $\be \in (-\infty, \be_{\star}]$ and all $t \in [a,b]$, the system 
\begin{equation} \label{E.nonlinearProjected}
\bL_*^{\perp}(\phi,\psi) - \bE_*^{\perp} - \bN_*^{\perp}(\phi,\psi) = 0\,, \quad (\phi,\psi) \in X_k \times \cK_{k,t}^{\perp}\,,
\end{equation}
has a unique solution $(\phi[t],\psi[t]) \in X_k \times \cK_{k,t}^{\perp}$ such that
\begin{equation} \label{E.boundNonlinear}
\|\phi[t]\| + \|\psi[t]\| \leq C \de_\be (1+|\be|)\,.
\end{equation}
Moreover, the map $t \mapsto (\phi[t], \psi[t])$ is continuously differentiable.
\end{proposition}

The proof of this result follows from a standard application of the contraction mapping theorem, as a consequence of the following result:

\begin{lemma} \label{L.linearCoupling}
For all fixed integer $k \geq 2$, all $D > 0$ and all $0 < a < b$, there exists $C > 0$ such that
\begin{align} \label{E.balltoball}
& \|\bN_*^{\perp}(\phi_0,\psi_0)\|=\|(\bN_*^{\perp}(\phi_0,\psi_0))_1\|+\|(\bN_*^{\perp}(\phi_0,\psi_0))_2\| \leq  \frac{C}{|\log\de_\be|}\Big(\|\phi_0\|+\|\psi_0\|\Big)  \,, \\ \label{E.contraction}
& \|\bN_*^{\perp}(\phi_1,\psi_1)-\bN_*^{\perp}(\phi_2,\psi_2)\| \leq  \frac{C}{|\log\de_\be|}  \Big( \|\phi_1-\phi_2\|+\|\psi_1-\psi_2\| \Big)\,,
\end{align}
for all $\be < -\sqrt{2}$, all $t \in [a,b]$ and all $(\phi_0,\psi_0),\,(\phi_1,\psi_1),\,(\phi_2,\psi_2) \in \dot{H}^1 \times \dot{H}^1$ satisfying
$$
\|\phi_i\| + \|\psi_i\| \leq D\, \de_\be \big(1+|\be|\big)\,.
$$
\end{lemma}

\begin{proof}
Firstly, note that there exists $C>0$ such that
\begin{equation*}\begin{split}
\|(\bN_*^{\perp}(\phi,\psi))_1\|& \leq C\|(U+\phi)_+^5-U^5-5U^4\phi+\beta(U+\phi)_+^2(V+\psi)_+^3-\beta U^2V^3\|_{\frac65}\,,\\
\|(\bN_*^{\perp}(\phi,\psi))_2\| &\leq C\|(V+\psi)_+^5-V^5-5V^4\psi+\beta(V+\psi)_+^2(U+\phi)_+^3-\beta V^2U^3\|_{\frac65}\,,
\end{split}\end{equation*}
for every $(\phi,\psi)\in \dot{H}^1 \times \dot{H}^1$. Also, notice that, doing Taylor and using the mean value theorem, we get the existence of $t,\lambda,\tau \in [0,1]$ such that
\begin{align*}
& \big|(U+\phi)_+^5-U^5-5U^4\phi\big|=\big|10(U+t\phi)_+^3\phi^2\big|\,, \\[0.25cm]
& \big|\beta(U+\phi)_+^2(V+\psi)_+^3-\beta U^2V^3\big| =\big|\beta\big((U+\phi)_+^2-U^2\big)(V+\psi)_+^3+\beta U^2\big((V+\psi)_+^3-V^3\big)\big|\\
& \quad  = \big|2\beta (U+\lambda\phi)_+\phi(V+\psi)_+^3+3\beta U^2(V+\tau\psi)^2_+\psi\big|\,,
\end{align*}
and hence
\begin{equation}\begin{split}\label{E.nonlinear1}
\|(\bN_*^{\perp}(\phi,\psi))_1\|&=O\Big(\big\|U^3\phi^2+|\phi|^5\big\|_{\frac65}\Big) +O\Big(|\beta|\big\|(U+|\phi|)|\phi|(V^3+|\psi|^3)+U^2(V^2+|\psi|^2)|\psi|)\big\|_{\frac65}\Big)\,.
\end{split}\end{equation}
Likewise, it can be seen that
\begin{equation}\begin{split}\label{E.nonlinear2}
\|(\bN_*^{\perp}(\phi,\psi))_2\|&=O\Big(\big\|V^3 \psi^2 +|\psi|^5\big\|_{\frac65}\Big) + O\Big(|\beta|\big\|(V+|\psi|)|\psi|(U^3+|\phi|^3)+V^2(U^2+|\phi|^2)|\phi|)\big\|_{\frac65}\Big)\,.
\end{split}\end{equation}
Since $\be \to -\infty$ here, some of these terms are not obvious to handle, being the worst ones the terms that are linear in $|\phi|$ and $|\psi|$. We focus on these terms and prove  that there exists a constant $C > 0$ such that
\begin{equation} \label{E.conclussionLinearTerms}
\begin{aligned}
& |\beta| \Big( \|U^2 V^2 \phi\|_{\frac65} + \|U^2 V^2 \psi\|_{\frac65} \Big) \leq  C  \frac{\sqrt{\de_\be}}{|\log\de_\be|^{\frac13}} \Big( \|\psi\| + \|\phi\| \Big)\,, \\
& |\beta| \Big( \|U V^3 \phi\|_{\frac65} + \|U^3 V \psi\|_{\frac65} \Big) \leq  \frac{C}{|\log\de_\be|}  \Big( \|\psi\| + \|\phi\| \Big)\,.
\end{aligned}  
\end{equation}
We estimate each term separately. Arguing as in the proof of Proposition \ref{P.error}, using twice again Lemma \ref{L.PremoselliVetois}, we get that
\begin{align*}
& |\be|^{\frac65} \|U^2 V^2 \psi\|_{\frac65}^{\frac65} \lesssim \|\psi\|^{\frac65} |\be|^{\frac65} \bigg( \int_{\RR^3} U^3 V^3 dx \bigg)^{\frac45} \\
& \quad \lesssim \|\psi\|^{\frac65} |\be|^{\frac65} \bigg( k \int_{\Om_1} U^{3} U_{t,1}^3 dx + k \int_{\Om_1} U^{3} \bigg( \sum_{j=2}^k U_{t,j} \bigg)^3 dx \bigg)^{\frac45} \lesssim \big[\|\psi\| |\be| \de_\be\big]^{\frac65} \Big(k^3 + k\log(k)^3 + k |\log\de_\be| \Big)^{\frac45}\,.
\end{align*}
Note that here, and during the rest of the proof, the implicit constants may depend on $a$ and $b$, but they are independent of $k$, $\beta$ and $t$. Since $k \geq 2$ is a fixed integer, we immediately deduce that there exists $C > 0$ (independent of $\be$ and $t$) such that
$$
|\be| \|U^2 V^2 \psi\|_{\frac65} \leq C \|\psi\| \de_\be |\be|  |\log\de_\be|^{\frac23} = C \|\psi\|\, \frac{\sqrt{\de_\be}}{|\log\de_\be|^{\frac13}}\,.
$$
The same proof also gives
$$
|\be| \|U^2 V^2 \phi\|_{\frac65} \leq C \|\phi\| \, \frac{\sqrt{\de_\be}}{|\log\de_\be|^{\frac13}}\,.
$$
Similarly, we get that
\begin{align*}
& |\be|^{\frac65} \|U V^3 \phi\|_{\frac65}^{\frac65} \lesssim |\be|^{\frac65} \|\phi\|^{\frac65} \bigg( k \int_{\Om_1} U^{6-\frac92} U_{t,1}^{\frac92} dx + k \int_{\Om_1} U^{6-\frac92} \bigg( \sum_{j=2}^k U_{t,j} \bigg)^{\frac92} \bigg)^{\frac45} \\
& \quad \lesssim \big[ \|\phi\|\sqrt{\de_\be}  |\be|  \big]^{\frac65} \Big( k+ \de_\be^{\frac32} (k^{\frac92} + k^{\frac52} \log(k)^{\frac92} ) \Big)^{\frac45}\,.
\end{align*}
Since $k \geq 2$ is a fixed integer, we deduce that there exists $C > 0$ (independent of $\be$ and $t$) such that
\begin{equation} \label{E.linearCouplingAux}
|\be| \|U V^3 \phi\|_{\frac65} \leq C \|\phi\|\sqrt{\de_\be}  |\be| =  \frac{C}{|\log \de_\be|} \|\phi\|\,.
\end{equation}
By using twice Lemma \ref{L.PremoselliVetois}, first with $\nu = 3/2$ and $\ga = 0$, and then with $\nu = 0$ and $\ga = 3/2$, we get
$$
|\be| \|U^3 V \psi\|_{\frac65} \leq  \frac{C}{|\log \de_\be|} \|\psi\|\,.
$$

The other terms in \eqref{E.nonlinear1} and \eqref{E.nonlinear2} can be seen to be smaller in a similar fashion and thus \eqref{E.balltoball} follows. Let us point out that the assumption $\|\phi\| + \|\psi\| \leq D \de_\be (1+|\be|)$ does not play any role in the proof of \eqref{E.conclussionLinearTerms}. However, when analyzing the nonlinear terms in \eqref{E.nonlinear1} and \eqref{E.nonlinear2}, it is needed.   

The rest of the proof is analogous so we skip the details. 
\end{proof}

\subsection{Energy expansion} Given a fixed integer $k \geq 2$, $t \in [a,b]$ and $\beta \in (-\infty, \be_\star]$, let $(\phi,\psi) = (\phi[t],\psi[t])$ be the unique solution to \eqref{E.nonlinearProjected} provided by Proposition \ref{P.nonlinear}. We consider the function
\begin{equation} \label{E.Fbeta}
F_\be: [a,b] \to \RR\,, \quad t \mapsto  \cJ_{\be}(U + \phi, V + \psi) \,,
\end{equation}
where
$$
\cJ_\be(u,v) := \frac12 \int_{\RR^3} \big( |\nabla u|^2 + |\nabla v|^2 \big)\, dx - \frac16 \int_{\RR^3} \big( u_+^6 + v_+^6 \big)\, dx - \frac{\be}{3} \int_{\RR^3} u_+^3 v_+^3 dx \quad \textup{for  } (u,v) \in \dot{H}^1 \times \dot{H}^1\,.
$$

This section is devoted to prove an expansion of $F_\be$ in terms of $\be$ and a function of $t$. As a preliminary step, we prove the following:

\begin{lemma} \label{L.firstExpansionEnergy}
For all fixed integer $k \geq 2$ and all $0 < a < b$, there exist constants $\be_\star < -\sqrt{2}$ and $C > 0$ such that, for all $\be \in (-\infty, \be_\star]$ and all $t \in [a,b]$, it holds
$$
F_\be(t) = \cJ_{\be}(U,V) + \te_{1,\be}(t)\,,  
$$
with $\te_{1,\be}  : [a,b] \to \RR$ satisfying
$$
|\te_{1,\be}(t)| \leq C \de_\be^2 (1+|\be|)^2 \quad \textup{for } t \in [a,b]\,.
$$
\end{lemma}

\begin{proof}
First note that, by the fundamental theorem of calculus,
$$
F_\be(t) - \cJ_\be(U,V) = \int_0^1  \frac{d}{ds} \cJ_\be (U+s\phi, V+s\psi)  ds\,.
$$
Moreover, using Propositions \ref{P.error} and \ref{P.nonlinear}, we get that
\begin{align*}
& \bigg|\frac{d}{ds} \cJ_\be (U+s\phi, V+s\psi)  \Big|_{s=0}\bigg| =\big| \pd_1 \cJ_\be(U,V) \phi + \pd_2 \cJ_\be(U,V) \psi\big| \\
& \quad \leq \|\cE_1 \phi\|_1 + \|\cE_2 \psi\|_1 \leq \|\cE_1\|_{\frac65} \|\phi\|_6 + \|\cE_2\|_{\frac65} \|\psi\|_6 \lesssim \|\cE_1\|_{\frac65} \|\phi\| + \|\cE_2\|_{\frac65} \|\psi\| \lesssim \de_\be^2 (1+|\be|)^2\,.
\end{align*}
Note that here and during the rest of the proof the implicit constants are independent of $\be$ and $t$. Then, using again the fundamental theorem of calculus and exchanging the order of integration, we get that
\begin{equation} \label{E.expansion1}
F_\be(t)- \cJ_\be(U,V) = \int_0^1 (1-s) \frac{d^2}{ds^2} \cJ_\be(U+s\phi,V+s\psi) ds + O(\de_\be^2 (1+|\be|)^2)\,.
\end{equation}

Now, using that $(\phi,\psi)$ solves \eqref{E.nonlinearProjected}, we get that
\begin{align*}
& \frac{d^2}{ds^2} \cJ_\be(U+s\phi,V+s\psi) = \int_{\RR^3} \big(  \cN_1(\phi,\psi) + \cE_1 \big) \phi\, dx  + \int_{\RR^3} \big( \cN_2(\phi,\psi) + \cE_2 \big) \psi\, dx \\
& \quad - 5 \int_{\RR^3} \big( (U+s\phi)_+^4 - U^4 \big) \phi^2 dx - 5 \int_{\RR^3} \big( (V+s\psi)_+^4 - V^4 \big) \psi^2\, dx - 6\be \int_{\RR^3} (U+s\phi)_+^2 (V+s\psi)_+^2 \phi\psi\, dx \\
& \quad - 2 \be \int_{\RR^3} (U+s\phi)_+(V+s\psi)_+^3 \phi^2\, dx - 2\be \int_{\RR^3} (U+s\phi)_+^3 (V+s\psi)_+ \psi^2\, dx\,.
\end{align*}
Arguing as in the proofs of Proposition \ref{P.error} and Lemma \ref{L.linearCoupling}, and using Propositions \ref{P.error} and \ref{P.nonlinear}, one can estimate each term on the right hand side and get a constant $C > 0$ (independent of $\be$ and $t$) such that
\begin{equation} \label{E.expansion2}
\bigg| \frac{d^2}{ds^2} \cJ_\be(U+s\phi,V+s\psi)  \bigg| \leq C \de_\be^2 (1+|\be|)^2\,.
\end{equation}
Note that, since $\be \to - \infty$, we need to be careful with the terms involving $\be$. However, following the proof of Lemma \ref{L.linearCoupling}, if is not difficult to see that \eqref{E.expansion2} holds true.

The result follows combining \eqref{E.expansion1} and \eqref{E.expansion2}. 
\end{proof}

\begin{proposition}\label{P.expansion2}
For all fixed integer $k \geq 2$ and all $0 < a < b$, there exist constants $\be_\star < -\sqrt{2}$ and $C > 0$ such that, for all $\be \in (-\infty,\be_\star]$ and all $t \in [a,b]$, it holds
$$
F_\be(t) = \frac{k+1}{3}\|U\|_6^6 + \de_\be \big(-\mathtt{c}_1 t + \mathtt{c}_2 t^\frac32 \big) + \te_{2,\be}(t)\,,
$$
with $\te_{2,\be}:[a,b] \to \RR$ satisfying 
$$
|\te_{2,\be}(t)| \leq C  \frac{\de_\be}{|\log\de_\be|}\,,\qquad \mathtt{c}_1:= \sqrt{6}\, \pi k \sum_{j=2}^k\Big[1-\cos\Big(\tfrac{2\pi(j-1)}{k}\Big)\Big]^{-\frac12},\qquad\mathtt{c}_2:=\sqrt{6}\, \pi k \,.$$
\end{proposition}

\begin{proof}
First note that we can decompose $\cJ_\be(U,V)$ as 
\begin{align*}
\cJ_\be(U,V) & = \frac13 \|U\|_6^6 + \frac12 \int_{\RR^3} |\nabla V|^2 dx - \frac16 \int_{\RR^3} V^6 \, dx - \frac{\be}{3} \int_{\RR^3} U^3 V^3 dx =: \frac13 \|U\|_6^6 + \cI_{1,\be}(V) + \cI_{2,\be}(U,V)\,.
\end{align*}
Having at hand this decomposition, we analyze each $\cI_{i,\be}$ separately. 

On one hand, following the proof of \cite[Lemma 3.4]{PrVe2019}, we see that
\begin{align*}
\cI_{1,\be}(V) & = \frac{k}{3} \|U\|_6^6  - \frac12 \sum_{j \neq i} \int_{\RR^3} U_{t,j}^5 U_{t,i} dx - \frac16 \int_{\RR^3} \bigg( V^6 - \sum_{j=1}^k U_{t,j}^6 - 6 \sum_{j \neq i} U_{t,j}^5 U_{t,i} \bigg) dx \\
& =  \frac{k}{3} \|U\|_6^6  - \mathtt{c}_1 t\, \de_\be + O(\de_\be^2) \,.
\end{align*}
On the other hand, taking into account the definition of $\Om_1$ given in \eqref{E.Omega1}, and using that $U$ and $V$ satisfy \eqref{E.invarianceRotation} and Lemma \ref{L.PremoselliVetois}, we get that
\begin{align*}
& \cI_{2,\be}(U,V)  = -\frac{\be k}{3} \int_{\Om_1} U^3 U_{t,1}^3 dx - \frac{\be k}{3} \int_{\Om_1} U^3 \bigg( \sum_{j=2}^k U_{t,j}\bigg)^3 dx\\
& \quad - \be k \int_{\Om_1} U^3 U_{t,1}  \bigg( \sum_{j=2}^k U_{t,j}\bigg)^2 dx - \be k \int_{\Om_1} U^3 U_{t,1}^2 \bigg( \sum_{j=2}^k U_{t,j} \bigg) dx  =  -\frac{\be k}{3} \int_{\Om_1} U^3 U_{t,1}^3 dx + O \bigg( \frac{\de_\be}{|\log\de_\be|} \bigg)\,. 
\end{align*}
Hence, taking into account Lemma \ref{L.firstExpansionEnergy}, to conclude the proof, we just have to show that  
\begin{equation} \label{E.conclussionEnergyExpansion}
-\frac{\be k}{3} \int_{\Om_1} U^3 U_{t,1}^3 dx = -\sqrt{6}\, k \pi \beta(\de_\be t)^{\frac32}|\log\de_\be|+O\Big(|\beta|\de_\be ^{\frac32}\Big) = \sqrt{6}\, k \pi \de_\be t^{\frac32} + O \bigg( \frac{\de_\be}{|\log\de_\be|} \bigg)\,.
\end{equation}

First of all, choosing $r_0>0$ small enough so that $B_{r_0}(\xi_{t,1}) \cap B_{r_0}(\xi_{t,j}) = \emptyset$ for all $j \neq 1$, and arguing as in the proof of \eqref{E.Zno1}, we get that
$$
-\frac{\be k}{3} \int_{\Om_1 \setminus B_{r_0}(\xi_{t,1})} U^3 U_{t,1}^3 dx = O \bigg( \frac{\de_\be}{|\log\de_\be|} \bigg)\,.
$$
On the other hand, observe that
\begin{equation} \label{E.conclussionEEsubs}
\begin{aligned}
& - \frac{\be k}{3} \int_{B_{r_0}(\xi_{t,1})} U^3 U_{t,1}^3 dx = -\frac{\be k}{3} (3t\de_\be)^\frac32 \int_{B_{\frac{r_0}{t\de_\be}}(0)} \frac{1}{(1+|t\de_\be y + \xi_{t,1}|^2)^\frac32} \frac{1}{(1+|y|^2)^\frac32} \, dy \\
& \quad = -\frac{\sqrt{3} \be k (t\de_\be)^\frac32}{(1+|\xi_{t,1}|^2)^\frac32}  \int_{B_{\frac{r_0}{t\de_\be}}(0)} \frac{dy}{(1+|y|^2)^\frac32} \\
& \qquad - \sqrt{3} \be k (t\de_\be)^\frac32 \int_{B_{\frac{r_0}{t\de_\be}}(0)} \bigg( \frac{1}{(1+|t\de_\be y + \xi_{t,1}|^2)^\frac32} - \frac{1}{(1+|\xi_{t,1}|^2)^\frac32} \bigg) \frac{dy}{(1+|y|^2)^\frac32}\,.
\end{aligned}
\end{equation}

Using polar coordinates, we get that
\begin{equation} \label{E.conclussionEEmain}
 \int_{B_{\frac{r_0}{t\de_\be}}(0)} \frac{dy}{(1+|y|^2)^\frac32} = 4\pi |\log \de_\be| + O(1)\,.
\end{equation}
Also, using that
$$
\frac{1}{(1+|t\de_\be y + \xi_{t,1}|^2)^\frac32} = \frac{1}{(1+|\xi_{t,1}|^2)^\frac32} + O(\de_\be^2 |y|^2 + 2 \de_\be y_1)\,,
$$
it follows that
\begin{equation} \label{E.conclussionEElower}
 - \sqrt{3} \be k (t\de_\be)^\frac32 \int_{B_{\frac{r_0}{t\de_\be}}(0)} \bigg( \frac{1}{(1+|t\de_\be y + \xi_{t,1}|^2)^\frac32} - \frac{1}{(1+|\xi_{t,1}|^2)^\frac32} \bigg) \frac{dy}{(1+|y|^2)^\frac32} = O \bigg( \frac{\de_\be}{|\log\de_\be|} \bigg)\,.
\end{equation}
By substituting \eqref{E.conclussionEEmain} and \eqref{E.conclussionEElower} into \eqref{E.conclussionEEsubs}, we get that \eqref{E.conclussionEnergyExpansion} holds true and the proof is concluded.  
\end{proof}

\begin{remark} \label{R.tstar}
For all fixed integer $k \geq 2$, the function 
$$
g: \RR \to \RR\,, \quad t \mapsto  - \mathtt{c}_1  t +  \mathtt{c}_2 t^\frac32\,,
$$
has a global minimum at 
\begin{equation} \label{E.tstar}
t_{\star} := \Big( \frac{2 \mathtt{c}_1}{3 \mathtt{c}_2} \Big)^2\,,
\end{equation}
which is non-degenerate (since $g''(t_\star)>0$) and so isolated.

\end{remark}

\subsection{Proof of Theorem \ref{T.mainResult2eq}} \label{S.proofTh21} We now have all the needed ingredients to prove Theorem \ref{T.mainResult2eq}. Here, given a fixed integer $k \geq 2$, $0 < a < b$, $t \in [a,b]$, $\be_\star < - \sqrt{2}$ as in Proposition \ref{P.nonlinear} and $\beta \in (-\infty, \be_\star]$, we let $(\phi,\psi) = (\phi[t],\psi[t])$ be the unique solution to \eqref{E.nonlinearProjected} provided by Proposition \ref{P.nonlinear}.

\begin{proof}[Proof of Theorem \ref{T.mainResult2eq}] Let $g \in C^1(\RR)$ be the function defined in Remark \ref{R.tstar}. We know that $g$ has a strict global minimum at a point $t_{\star} > 0$. It then follows from Proposition \ref{P.expansion2} that, for $|\be|$ sufficiently large, there exists a critical point $t_{\star,\be} > 0$ of the function $F_\be$ given in \eqref{E.Fbeta} such that $t_{\star,\be} \to t_{\star}$ as $\be \to - \infty$. We then define $(u_\be, v_\be) := (U + \phi[t_{\star,\be}], V[t_{\star,\be}] + \psi[t_{\star,\be}])$.

On one hand, since $(u_\be, v_\be)$ is a solution to \eqref{E.nonlinearProjected}, there exists $\cc(t_{\star,\be}) \in \RR$ such that
\begin{equation} \label{E.conclussionTh21}
\left\{
\begin{aligned}
& \pd_1 \cJ_\be(u_\be, v_\be) = 0\,, \\
& \pd_2 \cJ_\be(u_\be, v_\be) = \cc(t_{\star,\be}) \langle Z_{k,t}, \cdot \rangle\,.
\end{aligned}
\right.
\end{equation}
On the other hand, using that $t_{\star,\be} > 0$ is a critical point of $F_\be$, we obtain that
\begin{equation} \label{E.criticalPointFbeta}
\begin{aligned}
0 & = F_\be'(t_{\star,\be}) = \pd_1 \cJ_\be(u_\be, v_\be) \frac{d}{dt} \big( U + \phi \big)\Big|_{t=t_{\star,\be}} + \pd_2 \cJ_\be(u_\be, v_\be) \frac{d}{dt} \big( V + \psi \big)\Big|_{t=t_{\star,\be}} \\
& = \pd_2 \cJ_\be(u_\be, v_\be) \frac{d}{dt} \big( V + \psi \big)\Big|_{t=t_{\star,\be}} = \cc(t_{\star,\be}) \bigg\langle Z_{k,t_{\star,\be}}, \frac{d}{dt} \big( V + \psi \big)\Big|_{t=t_{\star,\be}}  \bigg\rangle \,.
\end{aligned}
\end{equation}

Now observe that, by direct computations,
$$
\frac{d}{dt} V = \frac{1}{t} Z_{k,t} - \frac{3^{\frac14} t^{\frac32} \de_\be^{\frac52}}{1-t^2\de_\be^2} \sum_{j=1}^k \frac{(x-\xi_{t,j})\cdot \xi_{t,j}}{(t^2 \de_\be^2 + |x-\xi_{t,j}|^2)^{\frac32}}\,.
$$
Thus, arguing as in \textit{Step 2} in the proof of Proposition \ref{P.linear}, we get that
\begin{equation} \label{E.criticalPointFbetaMainOrder}
\bigg\langle Z_{k,t_{\star,\be}}, \frac{d}{dt} V \Big|_{t=t_{\star,\be}} \bigg\rangle = \frac{15\sqrt{3}\pi^2 k}{64\, t_{\star,\be}}  + O(\de_\be)\,.
\end{equation}

Next, using that $\psi \in \cK_{k,t}^{\perp}$, we get that
$$
0 = \frac{d}{dt} \langle Z_{k,t}, \psi \rangle = \bigg\langle \frac{d}{dt} Z_{k,t}, \psi \bigg\rangle + \bigg\langle Z_{k,t}, \frac{d}{dt} \psi \bigg\rangle\,,
$$
and so that
$$
\bigg\langle Z_{k,t_{\star,\be}}, \frac{d}{dt} \psi \Big|_{t= t_{\star,\be}} \bigg\rangle = - \bigg\langle \frac{d}{dt} Z_{k,t}\Big|_{t = t_{\star,\be}}, \psi[t_{\star,\be}] \bigg\rangle\,.
$$
Hence, using Cauchy-Schwarz inequality and \eqref{E.boundNonlinear}, we get that
\begin{align*}
\bigg|  \bigg\langle Z_{k,t_{\star,\be}}, \frac{d}{dt} \psi \Big|_{t= t_{\star,\be}} \bigg\rangle \bigg| \leq \big\|\psi[t_{\star,\be}]\big\| \bigg\| \frac{d}{dt} Z_{k,t} \Big|_{t=t_{\star,\be}} \bigg\| \lesssim \de_\be (1+|\be|) \bigg\| \frac{d}{dt} Z_{k,t} \Big|_{t=t_{\star,\be}} \bigg\|\,.
\end{align*}
Moreover, direct but tedious computations show that
$$
\bigg\| \frac{d}{dt} Z_{k,t} \Big|_{t=t_{\star,\be}} \bigg\| \lesssim 5 \sum_{j=1}^k \bigg\| 4 U_{t_{\star,\be},j}^3 \Big(\frac{d}{dt} U_{t,j} \Big|_{t = t_{\star,\be}} \Big) Z_{t_{\star,\be} \de_\be, \xi_{t_{\star,\be},j}}^{(0)} + U_{t_{\star,\be},j}^4 \Big( \frac{d}{dt} Z_{t\de_{\be}, \xi_{t,j}}^{(0)} \Big|_{t = t_{\star,\be}} \Big) \bigg\|_{\frac65} = O(1)\,.
$$
We then have that
\begin{equation} \label{E.criticalPointFbetaRemainder}
\bigg|  \bigg\langle Z_{k,t_{\star,\be}}, \frac{d}{dt} \psi \Big|_{t= t_{\star,\be}} \bigg\rangle \bigg| = O \bigg( \frac{\sqrt{\de_\be}}{|\log \de_\be|} \bigg)\,.
\end{equation}
Gathering \eqref{E.criticalPointFbeta}--\eqref{E.criticalPointFbetaRemainder}, we get that
$$
0 = \cc(t_{\star,\be}) \left( \frac{15\sqrt{3}\pi^2 k}{64\, t_{\star,\be}} + O \bigg( \frac{\sqrt{\de_\be}}{|\log \de_\be|} \bigg) \right)\,,
$$
and so that, for $|\be|$ large enough, $\cc(t_{\star,\be}) = 0$. The result then follows combining this fact with \eqref{E.conclussionTh21}.
\end{proof}
 
\section{The case $m\ge3$} \label{S.generalcase}
The goal of this section is to prove Theorem \ref{T.mainResultIntro} in the general case $m\geq 3$. 
We consider the general system with $m=q+1 \geq 3$ components
\begin{equation} \label{1}
-\Delta u_i= u_i^5+ \sum\limits_{j=1\atop i\not=j}^{q+1} \beta_{ij}u_i ^2u_j^3\,, \quad u_i \gneqq 0 \quad \hbox{in}\ \mathbb R^3, \quad i \in \{1,...,q+1\}\,,\end{equation}
where $\beta_{ij}$ are coupling constants satisfying \eqref{E.beta}. Inspired by \cite{ChMePi2023}, the strategy will be to reduce this system into a new one of only two components, where we will apply a similar approach to that of Section \ref{S.m=2}. 

Let $ {\mathscr R}_\theta$ with $\theta\in\RR$ be given in \eqref{E.rotmatrix} and $r,k\in\NN$. As already mentioned in the Introduction, we set
\begin{equation}\label{E.rotation}
{\mathscr R}_{r,k}:={\mathscr R}_{\frac{(r-1)}{q}\frac{2\pi}{k}}\,,
\end{equation}
and consider the \textit{nonlocal} system 
\begin{equation}\label{2222}
\left\{
\begin{aligned}
&-\Delta u=u^5+\beta u^2 v^3 + \be u^2  \sum\limits_{r=2}^qv_r^3\quad && \hbox{in}\ \mathbb R^3\,,\\
&-\Delta v=v^5+\beta u^3 v^2+\alpha v^2 \sum\limits_{r=2}^qv_r^3\quad && \hbox{in}\ \mathbb R^3\,,
\end{aligned}
\right. \qquad \qquad u, v \gneqq 0 \quad \textup{in } \RR^3\,,
\end{equation}
where 
\begin{align*} 
v_r(x):=v({\mathscr R}_{r,k} x) \quad \hbox{for } r \in \{1,\dots,q\}\,.
\end{align*}
Then, we point out that, if $(u,v)$ is a solution to \eqref{2222} and satisfies 
\begin{equation}\label{E.sim}
u(x)=u({\mathscr R}_{r,k}x)\quad \hbox{for all } r \in \{2,\dots,q\}\,, \quad  \hbox{and}\quad  v(x)=v(\mathscr{R}_{\frac{2\pi}{k}}x)\,,
\end{equation}
the vector $(u_1,\dots, u_{q+1})$, with 
\begin{align*}
u_i(x)=v({\mathscr R}_{i,k} x)\quad \hbox{for all}\ i\in \{1,\dots,q\}\,, \quad \hbox{and}\quad u_{q+1}(x)=u(x)\,, 
\end{align*}
is a solution to \eqref{1}.  Indeed, by \eqref{E.sim}, we can adapt the computations in \cite[Section 3]{ChMePi2023} to get
\begin{equation}\label{E.reduction}
\sum\limits_{r=2}^qv_r^3({\mathscr R}_{i,k}x) = \sum\limits_{j=1\atop i\not=j}^{q}u_j^3(x) \quad \textup{for } i \in \{1, \ldots,q\}\,, 
\end{equation}
and therefore
\begin{equation*}
-\Delta u_i(x)=-\Delta v({\mathscr R}_{i,k} x)=\bigg[v^5+\beta v^2u^3+\alpha v^2 \sum\limits_{r=2}^qv_r^3\bigg]({\mathscr R}_{i,k} x)=\bigg[u_i^5+\beta u_i^2 u_{q+1}^3+\alpha u_i^2\sum\limits_{j=1\atop i\not=j}^{q}u_j^3\bigg](x)\,,
\end{equation*}
for $i \in \{1,\ldots,q\}$. Having at hand \eqref{2222}, we will prove the following result, from which Theorem \ref{T.mainResultIntro} in the case where $m \geq 3$ follows:
\begin{theorem}\label{T.mainResultmgeq3}
For any fixed integer $k \geq 2$ and any $\al \in \RR$, there exists $\beta_\star < -\sqrt{2}$ such that, for each $\beta \in (-\infty,\beta_\star]$,  \eqref{2222} has a finite energy solution $(u_\beta,v_\beta)$ of the form
$$
u_\beta = U + \phi_\beta \,, \quad v_\beta = \sum_{j=1}^k \frac{1}{\tilde{\de}_\be^\frac12}\, U \Big( \frac{\,\cdot\,-\xi_{\be,j}}{\tilde{\de}_\beta} \Big) + \psi_\beta\,.
$$
Here, $U$ is defined in \eqref{E.bubble}, $\phi_\beta, \psi_\beta \in \dot{H}^1$ and $(\xi_{\beta,j})_j \subset \RR^3$ satisfy
$$
\|\phi_\beta\|_{\dot{H}^1} + \|\psi_\beta\|_{\dot{H}^1} \to 0 \quad \textup{ and } \quad  \xi_{\beta,j} \to  \Big( \e^{\frac{2\pi(j-1)}{k}i}, 0 \Big)\,,   \quad \textup{as } \beta \to - \infty\,,
$$
and $\tilde{\de}_\beta \in (0,1/\e^2)$ satisfies 
$$
\tilde{\de}_\be^{\frac12} |\log \tilde{\de}_\be| \sim -\frac{1}{\beta}\,.
$$
\end{theorem}

\begin{remark}\label{R.3}
Some comments on Theorem \ref{T.mainResultmgeq3} are in order:
\begin{itemize}
\item[(i)]Since $U$ is a radial function, it trivially satisfies $U(x)=U({\mathscr R}_{r,k} x)$ for every $r \in \{ 2,\ldots,q\}$. Furthermore, it is straightforward to check that
$$ \widetilde{V}(x):= \sum_{j=1}^k \frac{1}{\tilde{\de}_\be^\frac12}\, U \Big( \frac{\,x-\xi_{\be,j}}{\tilde{\de}_\beta} \Big) \,,$$
satisfies that $\widetilde{V}(\mathscr{R}_{\frac{2\pi}{k}} x ) = \widetilde{V}(x)$ for all $x \in \RR^3$. Thus, if one finds $\phi_\beta,\psi_\beta$ in subspaces of $\dot{H}^1$ that respect the symmetries, then 
$$
u_{i,\beta}(x)=v_\beta ({\mathscr R}_{i,k} x) \quad \textup{for } i\in \{1,\ldots q\}\,,\quad u_{q+1,\beta}(x)=u_\beta(x)\,,
$$
solve the general system \eqref{1} and Theorem \ref{T.mainResultIntro} follows. \smallbreak
\item[(ii)] The choice of the rotation angle in \eqref{E.rotation} guarantees that $u_{i,\beta}\neq u_{j,\beta}$ for every $i\neq j$, $i,j \in \{1,\ldots, q+1\}$. Furthermore, there are no concentric bubbles in the construction; the minimum distance between two of them is of order $(qk)^{-1}$. \smallbreak
\item[(iii)] In \cite{ChMePi2023}, to reduce the system into a one of two equations,  even symmetry of $v$ is required. To ensure this, the authors need to impose $k$ even in the main theorem.  In our construction the reduction works only asking \eqref{E.sim} (see computation \eqref{E.reduction}), so no extra condition on $k$ is assumed.\smallbreak
\item[(iv)] System \eqref{2222} is different from \eqref{E.system2eq}, so a new proof is required. \smallbreak
\item[(v)] Analogously to the case $m=2$, since we are searching for nonnegative solutions we will deal with
\begin{equation}\label{E.problem+}\left\{\begin{aligned}
&-\Delta u=u_+^5+\beta u_+^2  \sum\limits_{r=1}^q(v_r)_+^3\quad \hbox{in}\ \mathbb R^3,\\
&-\Delta v=v_+^5+\beta v_+^2u_+^3+\alpha v_+^2 \sum\limits_{r=2}^q(v_r)_+^3\quad \hbox{in}\ \mathbb R^3,
\end{aligned}\right.
\end{equation}
where $s_+$ stands for the positive part. Indeed, it is straightforward to check that the solutions to this problem will solve \eqref{2222} (see \eqref{E.system2eq+} for more details).
\end{itemize} 
\end{remark}
As Theorem \ref{T.mainResultmgeq3} suggests,  we will look for a solution to \eqref{E.problem+}
 like \eqref{E.ansatz}, i.e.
 $u=U+\phi,\ v=V+\psi$ 
where $U$  and $V$ are defined in \eqref{E.sumofbubbles}, and $\phi$ and $\psi$ are small and belong to appropriate Banach spaces.  In order to inherit the symmetries in \eqref{E.sim} we need to define a new one:
$$\tilde{X}_k:=\big\{\phi\in \dot{H}^1(\RR^3):\phi \mbox{ satisfies }\eqref{E.invarianceEven},\eqref{E.invarianceRotation},\eqref{E.invarianceKelvin} \mbox{ and }\phi(x)=\phi ({\mathscr R}_{r,k} x)  \textup{ for } r \in \{2,\ldots, q\}\big\}\,.$$
\begin{remark}
Notice that $\tilde{X}_k$ is nothing but $X_{qk}$ (see \eqref{E.Xk}).
\end{remark}
We consider the orthogonal projections
$$
\tilde{\Pi}: \tilde{X}_k \times X_k \to \tilde{X}_k \times \cK_{k,t} \,, \qquad \tilde{\Pi}^{\perp}: \tilde{X}_k \times X_k \to \tilde{X}_k \times \cK_{k,t}^{\perp}\,,
$$
where $X_k$, $\cK_{k,t}$ and $\cK_{k,t}^{\perp}$ are given in \eqref{E.Xk} and \eqref{E.Kk}, and we rewrite \eqref{E.problem+} as the equivalent system
\begin{equation} \label{E.systemLyapunov-Schmidt3}
\left\{
\begin{aligned}
\ & \tilde{\Pi} \big[\bL^*(\phi,\psi) - \tilde{\bE}^* - \tilde{\bN}^*(\phi,\psi)\big] = 0 \,, \\
& \tilde{\Pi}^{\perp} \big[\bL^*(\phi,\psi) - \tilde{\bE}^* - \tilde{\bN}^*(\phi,\psi)\big] = 0 \,, 
\end{aligned}
\right.
\qquad (\phi,\psi) \in \tilde{X}_k \times X_k\,,
\end{equation}
where the error terms $\tilde{\bE}^*=(\tilde{\cE}_1^*, \tilde{\cE}_2^*)$ are given by
\begin{equation} \label{E.error3}
\tilde{\cE}_1^* := \cE_1^*+(-\Delta)^{-1}\Big(\beta U^2\sum_{r=2}^qV_r^3\Big), \quad \tilde{\cE}_2^* := \cE_2^*+(-\Delta)^{-1}\Big(\alpha V^2\sum_{r=2}^qV_r^3\Big)\,,
\end{equation}
and the coupling terms $\tilde{\bN}^*=(\tilde{\cN}^*_1,\tilde{\cN}^*_2)$ by
\begin{equation}\label{E.nonlinearCoupling3}
\begin{aligned}
& \tilde{\cN}^*_1(\phi,\psi) := \cN^*_1(\phi,\psi)+(-\Delta)^{-1}\Big(\beta(U+\phi)^2_+\sum_{r=2}^q(V_r+\psi_r)_+^3-\beta U^2\sum_{r=2}^qV_r^3\Big)\,, \\
& \tilde{\cN}^*_2(\phi,\psi) := \cN^*_2(\phi,\psi)+(-\Delta)^{-1}\Big(\alpha (V+\psi)_+^2\sum_{r=2}^q(V_r+\psi_r)_+^3-\alpha V^2\sum_{r=2}^q V_r^3\Big)\,.
\end{aligned}
\end{equation}
Here, we have set
$$V_r(x):=V({\mathscr R}_{r,k} x)\,, \qquad \psi_r(x):=\psi({\mathscr R}_{r,k} x)\,,$$
and $\bL^*, \,\cE_1^*,\,\cE_2^*,\,\cN^*_1,\,\cN^*_2$ are defined in \eqref{E.linear}--\eqref{E.nonlinearCoupling}. Furthermore, let us denote
$$U_{r,t,j}(x): = \frac{3^\frac14 \sqrt{t \de_\be}}{(t^2\de_\be^2 + |x-\xi_{r,t,j}|^2)^\frac12}\,, \quad \xi_{r,t,j} := \sqrt{1-t^2 \delta_\be^2}\, \Big( \e^{\left(\frac{2\pi(j-1)}{k}+\frac{(r-1)2\pi}{qk}\right)i}, 0 \Big)\,,$$
so that $V_r=\sum\limits_{j=1}^kU_{r,t,j}\,.$

Given the similarities, the proof of Theorem \ref{T.mainResultmgeq3} follows the same strategy as the case of $m=2$.  We will focus on the differences. Through the rest of this section $\al \in \RR$ is a fixed parameter, and the appearing constants may depend on $\al$.

\subsection{The error of approximation} 
We first see that the new terms appearing in the error are not larger than the ones we had in Section \ref{S.m=2} (see Proposiiton \ref{P.error}).

\begin{proposition} \label{P.error3}
For all fixed integer $k \geq 2$ and all $0 < a < b$, there exists a constant $C > 0$ such that
$$
\|\tilde{\cE}_1^*\| + \|\tilde{\cE}_2^*\| \leq C \de_\be (1+|\beta|)\,,
$$
for all $\beta < -\sqrt{2}$ and $t \in [a,b]$. 
\end{proposition}
\begin{proof}
Since $U$ is invariant under rotations, we have that
$$
\bigg\| \beta U^2\sum_{r=2}^q V_r^3 \bigg\|_{\frac65} \leq |\be|\sum_{r=2}^q \|U^2 V_r^3\|_{\frac65} = q |\be| \|U^2 V^3\|_{\frac65}\,,
$$
and hence, by Proposition \ref{P.error}, we get the estimate for $\tilde{\cE}_1^*$. On the other hand, 
\begin{equation*}\begin{split}
& \int_{\RR^3}\Big(\alpha V^2\sum_{r=2}^qV_r^3\Big)^\frac65\,dx = k\int_{\Omega_1}\Big(\alpha V^2\sum_{r=2}^qV_r^3\Big)^\frac65\,dx\leq k|\alpha|^\frac65\int_{\Omega_1}\Big(V^2\Big(\sum_{r=2}^qV_r\Big)^3\Big)^\frac65\,dx\\
& \quad \leq  k|2 \alpha|^\frac65\bigg[\int_{\Omega_1}\Big(U_{t,1}^2\Big(\sum_{r=2}^qV_r\Big)^3\Big)^\frac65\,dx+\int_{\Omega_1}\Big(\Big(\sum_{j=2}^kU_{t,j}\Big)^2\Big(\sum_{r=2}^qV_r\Big)^3\Big)^\frac65\,dx\bigg]\\
& \quad \leq k|2\alpha|^\frac65\bigg[\int_{\Omega_1}\Big(U_{t,1}^2\Big(\sum_{\ell=2}^{qk}\tilde{U}_{t,\ell}\Big)^3\Big)^\frac65\,dx+\int_{\Omega_1}\Big(\sum_{\ell=2}^{qk}\tilde{U}_{t,\ell}\Big)^6\,dx\bigg],
\end{split}\end{equation*}
with
$$
\tilde{U}_{t,\ell}(x): = \frac{3^\frac14 \sqrt{t \de_\be}}{(t^2\de_\be^2 + |x-\tilde{\xi}_{t,\ell}|^2)^\frac12} \quad \textup{and} \quad \tilde{\xi}_{t,\ell} := \sqrt{1-t^2 \delta_\be^2}\, \Big( \e^{\left(\frac{2\pi(\ell-1)}{qk}\right)i}, 0 \Big)\,, \quad \textup{for } \ell \in \{2, \ldots, qk\}\,.
$$
Keeping in mind Remark \ref{R.3}, (iii), $\sum_{\ell=2}^{qk}\tilde{U}_{t,\ell}$ can be seen as a generalization of the usual construction but with $qk$ peaks. Since both $q$ and $k$ are fixed we can apply Lemma \ref{L.PremoselliVetois} to both terms, with $\nu=12/5$, $\ga=18/5$ and $\nu=0$, $\ga=6$ respectively,  to conclude that
$$
\bigg\|\alpha V^2\sum_{r=2}^qV_r^3\bigg\|_{\frac65} = O\big(\delta_\beta^\frac52\big)\,.
$$
The estimate for $\tilde{\cE}_2^*$ then follows by Proposition \ref{P.error} and so does the result. 
\end{proof}

\subsection{The linear theory}
Denote
$$
\tilde{\bL}^{\perp}_* := \tilde{\Pi}^{\perp} \bL^*\,, \quad \tilde{\bE}^{\perp}_* := \tilde{\Pi}^{\perp} \tilde{\bE}^*\,, \quad \textup{and} \quad \tilde{\bN}^{\perp}_* := \tilde{\Pi}^{\perp} \tilde{\bN}^*\,.
$$
Then, we have the following:

\begin{proposition} \label{P.linear3}
For all fixed integer $k \geq 2$ and all $0 < a < b$, there exist constants $\be_\star < -\sqrt{2}$ and $C > 0$ such that, for all $\be \in (-\infty, \be_\star]$ and all $t \in [a,b]$,
\begin{equation}
\|(\phi,\psi)\| \leq C \|\tilde{\bL}^{\perp}_*(\phi,\psi)\| \,, \quad \textup{\textit{for all} } (\phi,\psi) \in \tilde{X}_k \times \cK_{k,t}^{\perp}\,.
\end{equation}
In particular, the inverse operator $(\tilde{\bL}^{\perp}_*)^{-1} : \tilde{X}_k \times \cK_{k,t}^{\perp} \to \tilde{X}_k \times \cK_{k,t}^{\perp}$ exists and is continuous. 
\end{proposition}

\begin{proof}
The proof follows exactly as in Proposition \ref{P.linear} just by noticing that the operator $\bL^*$ preserves the extra symmetry appearing in $\tilde{X}_k$. That is, using the fact that $U$ is radial, one gets that $\bL^*$ maps $\tilde{X}_k\times X_k$ into $\tilde{X}_k\times X_k$. 
Since $\tilde{X}_k\subset X_k$, the rest of the proof can be analogously reproduced.
\end{proof}

\subsection{Nonlinear theory} Having at hand Propositions \ref{P.error3} and Proposition \ref{P.linear3} we can now give the required existence result for 
$$\tilde{\Pi}^{\perp} \big[\bL^*(\phi,\psi) - \tilde{\bE}^* - \tilde{\bN}^*(\phi,\psi)\big] = 0 \,, 
\qquad (\phi,\psi) \in \tilde{X}_k \times X_k\,.$$
Indeed, we have the following:

\begin{proposition} \label{P.nonlinear3}
For all fixed integer $k \geq 2$ and all $0 < a < b$, there exist constants $\be_{\star} < -\sqrt{2}$ and $C > 0$ such that, for all $\be \in (-\infty, \be_{\star}]$ and all $t \in [a,b]$, the system 
\begin{equation} \label{E.nonlinearProjected3}
\tilde{\bL}^{\perp}_*(\phi,\psi) - \tilde{\bE}^{\perp}_* - \tilde{\bN}^{\perp}_*(\phi,\psi) = 0\,, \quad (\phi,\psi) \in \tilde{X}_k \times \cK_{k,t}^{\perp}\,,
\end{equation}
has a unique solution $(\phi[t],\psi[t]) \in \tilde{X}_k \times \cK_{k,t}^{\perp}$ such that
$$
\|\phi[t]\| + \|\psi[t]\| \leq C \de_\be (1+|\be|)\,.
$$
Moreover, the map $t \mapsto (\phi[t], \psi[t])$ is continuously differentiable.
\end{proposition}

\begin{proof}
As stated in Section \ref{S.m=2}, once we have Propositions \ref{P.error3} and Proposition \ref{P.linear3}, this result follows by 
a standard fixed point argument. The fact that the nonlinear term is contractive follows in the same spirit as Lemma \ref{L.linearCoupling}. Hence, it only remains to check that $\tilde{\bE}^{\perp}_* + \tilde{\bN}^{\perp}_*$ maps $\tilde{X}_k \times \cK_{k,t}^{\perp}$ into itself. 

By construction, for every $r \in \{1,\ldots q\}$, $V_r$ satisfies the symmetries \eqref{E.invarianceRotation}, \eqref{E.invarianceKelvin},  and 
$$
V_r(x_1,x_2,x_3)=V_r(x_1,x_2,-x_3)\,.
$$
Also, using that $V(x_1,x_2,x_3) = V(x_1,-x_2,x_3)$,  it can be seen that
$$
V_r(x_1,x_2,x_3)=V_{q-r+2}(x_1,-x_2,x_3) \quad \textup{for } r\in \{1,\ldots q\}\,,
$$
and hence in particular
$$
\sum_{r=2}^q V_r(x_1,x_2,x_3)^\gamma=\sum_{r=2}^qV_r(x_1,-x_2,x_3)^\gamma,\quad \gamma\in\NN\,.
$$
Likewise,  if $\psi\in X_k$ then $\psi_r(x_1,x_2,x_3)=\psi_{q-r+2}(x_1,-x_2,x_3)$ and hence
$$
\sum_{r=2}^q \big([V_r+\psi_r](x_1,x_2,x_3)\big)^\gamma=\sum_{r=2}^q \big([V_r+\psi_r](x_1,-x_2,x_3)\big)^\gamma,\quad \gamma\in\NN\,.
$$
Having at hand this information, it is straightforward to see that $\tilde{\bE}^*+\tilde{\bN}^*(\phi,\psi)$ satisfies \eqref{E.invarianceEven}--\eqref{E.invarianceKelvin} whenever $(\phi,\psi)\in \tilde{X}_k\times X_k$. 

On the other hand, noticing that
$$
V_r({\mathscr R}_{r,k}x)=V_{r+1}(x) \quad \textup{for } r \in \{2, \ldots,q\}\,,
$$
we get that, for every pair $(\phi,\psi)\in \tilde{X}_k\times X_k$,
$$
\Big[\tilde{\cE}^*_1+\tilde{\cN}^*_1(\phi,\psi)\Big]({\mathscr R}_{r,k}x)= \Big[\tilde{\cE}^*_1+\tilde{\cN}^*_1(\phi,\psi)\Big](x) \quad  \textup{for } r \in \{2, \ldots,q\}\,,
$$
and the result easily follows. 
\end{proof}

\subsection{Finite dimensional reduction} Given a fixed integer $k \geq 2$, $0<a<b$, $t \in [a,b]$, $\be_\star < - \sqrt{2}$ as in Proposition \ref{P.nonlinear3} and $\beta \in (-\infty, \be_\star]$, we let $(\phi,\psi) = (\phi[t],\psi[t])$ be the unique solution to  \eqref{E.nonlinearProjected3} provided by Proposition \ref{P.nonlinear3}.  The goal of this subsection is to find $t$ in such a way that $(\phi,\psi)$ solves the complete problem \eqref{E.systemLyapunov-Schmidt3}. In Section \ref{S.m=2} we faced this by a variational approach, which is not available here. We will use the so-called direct method: we rewrite the first equation in \eqref{E.systemLyapunov-Schmidt3}  as
\begin{equation}\label{E.fine1}  
\mathcal L_2^*(\psi) - \tilde{\mathcal E}_2^* - \tilde{\mathcal N}_2^*(\phi,\psi)  = c_0(t) Z_{k,t}\,,
\end{equation}
where the function $Z_{k,t}$ is given in \eqref{E.Z}, and hence we want to find $t$ so that $c_0(t)=0$ or, equivalently (see \eqref{E.ZnExpansion}), so that
\begin{equation*}\label{E.proj3}
\tilde{\cc}(t):=\int_{\RR^3}\nabla (\mathcal L_2^*(\psi) - \tilde{\mathcal E}_2^* - \tilde{\mathcal N}_2^*(\phi,\psi)) \nabla Z_{k,t}\,dx=0\,.
\end{equation*}

\begin{proposition}\label{P.reduction3}
For all fixed integer $k \geq 2$ and all $0 < a < b$, there exist constants $\be_\star < -\sqrt{2}$ and $C > 0$ such that, for all $\be \in (-\infty,\be_\star]$ and all $t \in [a,b]$, it holds
$$
\tilde{\cc}(t)=-\tilde{\mathtt c}_1 t\delta_\beta-\tilde{\mathtt c}_2\beta (t\delta_\beta)^\frac32|\log \delta_\beta|+\theta_{3,\beta}(t)\,,
$$
where $ \theta_{3,\beta}:[a,b]\to\RR$ satisfies
$$
|\theta_{3,\beta}(t)|\leq C \de_\be^{2} (1+|\be|)^2\, ,\qquad \tilde{\mathtt{c}}_1:= \sqrt{6}\, \pi k \sum_{j=2}^k\Big[1-\cos\Big(\tfrac{2\pi(j-1)}{k}\Big)\Big]^{-\frac12},\qquad\tilde{\mathtt{c}}_2:=\frac32\sqrt{6}\, \pi k\,.
$$

\end{proposition}

\begin{proof}
First of all, it can be seen that 
\begin{equation}\label{E.projection3}
\int_{\RR^3}\nabla (\mathcal L_2^*(\psi) - \tilde{\mathcal E}_2^* - \tilde{\mathcal N}_2^*(\phi,\psi)) \nabla Z_{k,t}\,dx= -\int_{\RR^3}\nabla  \tilde{\mathcal E}_2^* \nabla Z_{k,t}\,dx + O \big(\de_\be^{2} (1+|\be|)^2 \big) \,.
\end{equation}
Indeed, using the fact that 
\begin{equation}\label{E.boundZ}
|Z^{(0)}_{t\delta_\beta,\xi_{t,j}}|\leq CU_{t\delta_\beta,\xi_{t,j}}\quad \mbox{ and }\quad |Z_{k,t}|\leq CV\quad \mbox{ for some constant }C>0\,,
\end{equation}
the problem satisfied for $Z_{k,t}$ and Proposition \ref{P.nonlinear3}, we can estimate
$$
\int_{\RR^3}\nabla \mathcal L_2^*(\psi) \nabla Z_{k,t}\,dx=\int_{\RR^3}\mathcal L_2(\psi) Z_{k,t}\,dx=5\int_{\RR^3}\Big(V^4-\sum_{j=1}^kU_{t,j}^4\Big)\psi Z_{k,t}\,dx=O \big( \de_\be^2 (1+|\be|) \big)\,.
$$
Also, proceeding as in the proof of Lemma \ref{L.linearCoupling} and using Proposition \ref{P.nonlinear3}, one can see that
\begin{equation*}
\begin{split}
& \int_{\RR^3}\nabla \tilde{\mathcal N}_2^*(\phi,\psi)\nabla Z_{k,t}\,dx=\int_{\RR^3}\tilde{\mathcal N}_2(\phi,\psi)Z_{k,t}\,dx\\
&\quad \lesssim \bigg|\beta\int_{\RR^3}(V^2U^2\phi+U^3V\psi )Z_{k,t}\,dx\bigg| + |\alpha| \sum_{r=2}^q \bigg| \int_{\RR^3}(V^2V_r^2\psi_r+V V_r^3\psi)Z_{k,t}\,dx\bigg|+O(\delta_\beta^2(1+|\beta|)^2)\,,
\end{split}
\end{equation*}
and that
$$
\bigg|\beta\int_{\RR^3}(V^2U^2\phi+U^3V\psi )Z_{k,t}\,dx\bigg| \leq |\beta|(\|\psi\|+\|\phi\|)\Big(\|V^3U^2\|_{\frac65}+\|V^2U^3\|_{\frac65}\Big)= O \big( \de_\be^{2} (1+|\be|)^2 \big)\,.
$$
Furthermore,
$$|\alpha| \sum_{r=2}^q \bigg| \int_{\RR^3}(V^2V_r^2\psi_r+V V_r^3\psi)Z_{k,t}\,dx\bigg| \leq |\alpha|\|\psi\|\Big(\|V^3V_r^2\|_{\frac65}+\|V^2V_r^3\|_{\frac65}\Big) = O  \big( \de_\be^3 (1+|\be|) \big)\,,$$
and thus \eqref{E.projection3} follows. Notice that the key point to get this estimate is the fact that 
\begin{equation}\label{E.separation}
\mbox{there exists }\eta_0>0\mbox{ such that }|\xi_{r,t,j}-\xi_{s,t,i}|>\eta_0 \mbox{ whenever }i\neq j\mbox{ or }r\neq s.
\end{equation}
Now, let us estimate the projection of the error term:
\begin{equation*}
\begin{split}
& \int_{\RR^3}\nabla \tilde{\mathcal E}_2^* \nabla Z_{k,t}\,dx=\int_{\RR^3}\tilde{\mathcal E}_2Z_{k,t}\,dx\\
& \quad =\underbrace{\int_{\RR^3}\Big(V^5-\sum_{j=1}^kU_{t,j}^5\Big)Z_{k,t}\,dx}_{I_1}+\underbrace{\beta\int_{\RR^3}U^3V^2Z_{k,t}\,dx}_{I_2}+\underbrace{\alpha\int_{\RR^3}V^2\Big(\sum_{r=2}^q V_r\Big)^3Z_{k,t}\,dx}_{I_3}\,.
\end{split}
\end{equation*}
We estimate each $I_i$, $i \in \{1,2,3\}$ separately. 

First, using the symmetries of the construction and \eqref{E.boundZ} once again, we get that
\begin{equation*}
\begin{split}
I_1&=k\int_{\Omega_1}\Big(V^5-\sum_{j=1}^kU_{t,j}^5\Big)Z_{k,t}\,dx\\
&=5k\int_{\Omega_1}U_{t,1}^4\Big(\sum_{j=2}^kU_{t,j}\Big)Z^{(0)}_{t\delta_\beta,\xi_{t,1}}\,dx+O\Big(\sum_{\gamma=0}^4\int_{\Omega_1}U_{t,1}^\gamma\Big(\sum_{j=2}^kU_{t,j}\Big)^{6-\gamma}\, dx\Big)\\
&=5k\int_{\Omega_1}U_{t,1}^4\Big(\sum_{j=2}^kU_{t,j}\Big)Z^{(0)}_{t\delta_\beta,\xi_{t,1}}\,dx+O(\delta_\beta^2)\,.
\end{split}
\end{equation*}
Note that in the last step we have used once again Lemma \ref{L.PremoselliVetois}. Now, let $r_0>0$ small enough such that $B_{r_0}(\xi_{t,1})\cap B_{r_0}(\xi_{t,j})=\emptyset$ for every $j\neq 1$. Then,
\begin{equation*}\begin{split}
& \int_{\Omega_1}U_{t,1}^4\Big(\sum_{j=2}^kU_{t,j}\Big)Z^{(0)}_{t\delta_\beta,\xi_{t,1}}\,dx=\int_{B_{r_0}(\xi_{t,1})}U_{t,1}^4\Big(\sum_{j=2}^kU_{t,j}\Big)Z^{(0)}_{t\delta_\beta,\xi_{t,1}}\,dx+O(\delta_\beta^3)\\
& \quad =t^{\frac12}\delta_\beta^{\frac12}\int_{B_{\frac{r_0}{\delta_\beta}}(0)}U^4(y)\Big(\sum_{j=2}^kU_{t,j}(\delta_\beta y+\xi_{t,1})\Big)Z^{(0)}(y)\,dy=t\delta_\beta\sum_{j=2}^k\frac{3^{\frac14}}{|\xi_{t,j}-\xi_{t,1}|}\int_{B_{\frac{r_0}{\delta_\beta}}(0)}U^4Z^{(0)}\,dy+O(\delta_\beta^2)\,,
\end{split}\end{equation*}
and thus
$$
I_1=\tilde{\mathtt c}_1 t\delta_\beta+O(\delta_\beta^2)\,.
$$

Next, by \eqref{E.boundZ} and Lemma \ref{L.PremoselliVetois}, we get that
\begin{equation*}\begin{split}
I_2&=\beta k\int_{\Omega_1}U^3V^2 Z_{k,t}\,dx=\beta k\int_{\Omega_1}U^3V^2Z^{(0)}_{t\delta_\beta,\xi_{t,1}}\,dx +O\bigg(|\beta|\sum_{\gamma=0}^2\int_{\Omega_1}U^3U_{t,1}^\gamma\Big(\sum_{j=2}^kU_{t,j}\Big)^{3-\gamma}\,dx\bigg)\\
&=\beta k\int_{\Omega_1}U^3U_{t,1}^2Z^{(0)}_{t\delta_\beta,\xi_{t,1}}\,dx +O\bigg(|\beta|\sum_{\gamma=0}^2\int_{\Omega_1}U^3U_{t,1}^\gamma\Big(\sum_{j=2}^kU_{t,j}\Big)^{3-\gamma}\,dx\bigg)\\
&=\beta k\int_{B_{r_0}(\xi_{t,1})}U^3U_{t,1}^2Z^{(0)}_{t\delta_\beta,\xi_{t,1}}\,dx +O\Big(|\beta|\delta_\beta^{\frac32}\Big)\,.
\end{split}\end{equation*}
On the other hand, using \cite[Equation (2.11)]{medinamusso2021}, we get that
\begin{equation*}\begin{split}
& \int_{B_{r_0}(\xi_{t,1})} U^3U_{t,1}^2Z^{(0)}_{t\delta_\beta,\xi_{t,1}}\,dx=\delta_\beta^{\frac32}\int_{B_{\frac{r_0}{\delta_\beta}}(0)}U^3(\delta_\beta y+\xi_{t,1})U^2(y)Z^{(0)}(y)\,dy\\
& \quad =\delta_\beta^{\frac32}U^3(\xi_{t,1})\bigg[\int_{B_{\frac{r_0}{\delta_\beta}}(0)}U^2(y)Z^{(0)}(y)\,dy+O(1)\bigg] =\delta_\beta^{\frac32}U^3(0,0,1)(1+O(\delta_\beta^2))\bigg[\frac{3^\frac34}{2}\int_{B_{\frac{r_0}{\delta_\beta}}(0)}\frac{dy}{(1+|y|^2)^{\frac32}}+O(1)\bigg]\,.
\end{split}\end{equation*}
and so, by \eqref{E.conclussionEEmain}, we conclude that
$$
I_2=\tilde{\mathtt c}_2 \beta (t\delta_\beta)^\frac32|\log\delta_\beta|+O\Big(|\beta| \delta_\beta^\frac32\Big)\,\,.
$$

Finally, using \eqref{E.boundZ} and \eqref{E.separation}, it is straightforward to check that
$$
|I_3|\leq \Big|\alpha\int_{\RR^3}V^3\Big(\sum_{r=2}^q V_r\Big)^3\,dx\Big| = O \big( \delta_\beta^3|\log\delta_\beta|\big)\,.
$$
Substituting the estimates for $I_i$, $i \in \{1,2,3\}$, into \eqref{E.projection3} the result follows.
\end{proof}

\subsection{Proof of Theorem \ref{T.mainResultmgeq3}} We now have all the needed ingredients to prove Theorem \ref{T.mainResultmgeq3}. Here, given a fixed integer $k \geq 2$, $0 < a < b$, $t \in [a,b]$, $\be_\star < - \sqrt{2}$ as in Proposition \ref{P.nonlinear3} and $\beta \in (-\infty, \be_\star]$, we let $(\phi,\psi) = (\phi[t],\psi[t])$ be the unique solution to \eqref{E.nonlinearProjected3} provided by Proposition \ref{P.nonlinear3}.

\begin{proof}[Proof of Theorem \ref{T.mainResultmgeq3}]
First of all, choosing $\delta_\beta$ such that
$$\delta_\beta^\frac12|\log\delta_\beta|=-\frac{1}{\beta}\,,$$
from Proposition \ref{P.reduction3}, it follows that
$$
\tilde{\cc}(t)=\Big(-\tilde{\mathtt c}_1t+\tilde{\mathtt c}_2 t^{\frac32}\Big) \delta_\beta+\theta_{3,\beta}(t)\,,
$$
with $\tilde{\cc}(t)$ defined in \eqref{E.proj3}. Hence, we can choose
$$\tilde{t}_{*,\be}= \Big(\frac{\tilde{\mathtt c}_1}{\tilde{\mathtt c}_2}\Big)^2+o_{\delta_\beta}(1)\,,$$
to have $\tilde{\cc}(\tilde{t}_{*,\be})=0$. Fixing $a,b$ such that $0<a<t_*<b$, we conclude that
\begin{equation*}
\mathcal L_2^*(\phi,\psi) - \tilde{\mathcal E}_2^* - \tilde{\mathcal N}_2^*(\phi,\psi)=0\,,
\end{equation*}
and therefore $(\phi_{\delta_\beta},\psi_{\delta_\beta}) = (\phi[\tilde{t}_{*,\be}],\psi[\tilde{t}_{*,\be}])$ solves \eqref{E.systemLyapunov-Schmidt3}.
\end{proof}

\section*{Acknowledgments}

A.J.F.  and M.M. are partially supported by the grant PID2023-149451NA-I00 of MCIN/AEI/10.13039/ 501100011033/FEDER, UE  and by Proyecto de Consolidaci\'on Investigadora, CNS2022-135640, MICINN (Spain).  M.M. is also supported by RYC2020-030410-I, MICINN (Spain). A.P. acknowledges support of INDAM-GNAMPA project “Problemi di doppia curvatura su variet`a a bordo e legami con le EDP di tipo ellittico” and of the project “Pattern formation in nonlinear phenomena” funded by the MUR Progetti di Ricerca di Rilevante Interesse Nazionale (PRIN) Bando 2022 grant 20227HX33Z.

\bibliography{system.bib}
\bibliographystyle{abbrv}

\vspace{0.5cm}

\end{document}